\documentclass[a4paper, 12pt]{article}

\usepackage{comment}
\usepackage{graphicx}
\usepackage{amsmath}
\usepackage{amssymb}
\usepackage{amsthm}
\usepackage{pxfonts}
\usepackage{enumerate}
\usepackage{color}
\usepackage{mathdots}
\usepackage{sectsty}
\usepackage[hidelinks]{hyperref}

\sectionfont{\scshape\centering\fontsize{12}{14}\selectfont}
\subsectionfont{\scshape\fontsize{12}{14}\selectfont}
\usepackage{fancyhdr}

\newcommand\shorttitle{The boundary of the orbital beta process}
\newcommand\authors{Theodoros Assiotis and Joseph Najnudel}

\newtheorem{thm}{Theorem}[section]

\newtheorem{lem}[thm]{Lemma}
\newtheorem{defn}[thm]{Definition}
\newtheorem{rmk}[thm]{Remark}
\newtheorem{prop}[thm]{Proposition}
\newtheorem{ex}[thm]{Example}

\title{\large \bf THE BOUNDARY OF THE ORBITAL BETA PROCESS}
\author{\small THEODOROS ASSIOTIS AND JOSEPH NAJNUDEL}
\date{}

\begin{document}

\maketitle

\begin{abstract}
The unitarily invariant probability measures on infinite Hermitian matrices have been classified by Pickrell \cite{Pickrell} and by Olshanski and Vershik \cite{OlshanskiVershik}. This classification is equivalent to determining the boundary of a certain inhomogeneous Markov chain with given transition probabilities. This formulation of the problem makes sense for general $\beta$-ensembles when one takes as the transition probabilities the Dixon-Anderson conditional probability distribution \cite{Dixon},\cite{Anderson}. In this paper we determine the boundary of this Markov chain for any $\beta \in (0,\infty]$, also giving in this way a new proof of the classical $\beta=2$ case of \cite{Pickrell},\cite{OlshanskiVershik}. Finally, as a by-product of our results we obtain alternative proofs of the almost sure convergence of the rescaled Hua-Pickrell and Laguerre $\beta$-ensembles to the general $\beta$ Hua-Pickrell \cite{KillipStoiciu}, \cite{ValkoVirag} and $\beta$ Bessel \cite{RiderRamirez} point processes respectively.
\end{abstract}

\tableofcontents

\section{Introduction}
The probability measures on infinite Hermitian matrices which are invariant by unitary conjugation have been completely classified by Pickrell in 
  \cite{Pickrell} and by Olshanski and Vershik in \cite{OlshanskiVershik}. These measures can be decomposed as convex combinations of extremal measures, called {\it ergodic measures}, which are indexed by a set of parameters 
$(\{\alpha^+\}, \{\alpha^-\}, \gamma_1, \gamma_2 ) \in \mathbb{R}_+^{\infty} \times  \mathbb{R}_+^{\infty} 
\times \mathbb{R} \times \mathbb{R}_+ $. 
Moreover, in  \cite{BorodinOlshanski}, Borodin and Olshanski have proven that  the points $(\alpha^+_j)_{j \geq 1}$
and $(-\alpha^-_j)_{j \geq 1}$  correspond
to almost sure limits of the extremal eigenvalues of the top-left submatrices of the corresponding infinite matrix, divided by their dimension. 
Besides this result of convergence of renormalized eigenvalues, a result of strong convergence of components of eigenvectors has recently been proven by Najnudel in \cite{Najnudel}, and previously by Maples, Najnudel and Nikeghbali in \cite{MaplesNajnudelNikeghbali} 
in the significant particular case of the Hua-Pickrell measure of parameter $0$, for which the image of the top-left submatrices by the Cayley transform are distributed like  the Circular Unitary Ensemble. 
The Cayley transform is the map from the Hermitian to the unitary matrices, given by $M \mapsto (M+i) (M-i)^{-1}$. This transform maps top-left blocks of infinite Hermitian matrices to 
particular sequences of unitary matrices called {\it virtual isometries}, defined by Neretin in   \cite{NeretinUnitary} and extended by Bourgade, Najnudel and Nikeghbali in 
 \cite{BourgadeNajnudelNikeghbali}. This extension includes some particular sequences of permutation matrices, corresponding to the so-called {\it virtual permutations}, for which a classification of the conjugation-invariant measures has been studied by Kerov, Olshanski and Vershik in \cite{KerovOlshanskiVershik} and by Tsilevich in \cite{Tsilevich}. 

The main goal of this paper is to generalize the classification of Pickrell, Olshanski and Vershik to the setting of $\beta$-ensembles for general parameter $\beta \in (0, \infty]$. The case $\beta=2$ corresponds to the infinite Hermitian matrices already considered above. The case $\beta = 1$ corresponds to infinite orthogonal matrices, and  $\beta =4$ corresponds to infinite self-adjoint matrices with quaternion entries. The case of other values of $\beta$ does not correspond to classical ensembles of matrices. 

However, the generalization can be naturally constructed if we only consider the spectra of the top-left submatrices. Then, the classification of Pickrell and Olshanski and Vershik becomes equivalent to the problem of determining the boundary of a certain inhomogeneous Markov chain (which was called 'the graph of spectra' by Kerov, see Section 8 in \cite{BorodinOlshanski} and also \cite{HuaPickrell} for more on this point of view), that we will make precise in the rest of the introduction. Such problems, see for example \cite{BorodinOlshanskiBoundary}, \cite{OkounkovOlshanskiJack}, \cite{CuencaJack}, \cite{CuencaMacdonald} (with the important difference to our setting that the state spaces are discrete) have significant applications in representation theory, beginning with the work of Vershik and Kerov on the infinite-dimensional unitary \cite{VershikKerov} and symmetric groups \cite{VershikKerovSymmetric}, and in more recent years in interacting particle systems and random surface growth, see \cite{BorodinKuanPlancherel}, \cite{BorodinKuanOrthogonal}, \cite{BorodinFerrari}. Coming back to the setting of this work, our results have some interesting consequences in the study of $\beta$-ensembles, see \cite{KillipStoiciu}, \cite{ValkoVirag}, \cite{RiderRamirez}, in that they provide an alternative route (without making use of tridiagonal or CMV matrix models) to proving the almost sure convergence of the rescaled Hua-Pickrell and Laguerre general $\beta$-ensembles. 

In order to state the main results of this paper precisely, we need to introduce some definitions.  First, a remark about the notation: throughout this paper we will use the parameter $\theta=\beta/2$. This is because we will make substantial use of symmetric functions in our argument, more precisely the multivariate Bessel functions, and the choice of the parameter $\theta=\beta/2$ is standard in the corresponding literature.

For $N\ge 1$, we consider the Weyl chambers $W^N$ given by: 
\begin{align}
W^N=\{a=(a_1,\dots,a_N) \in \mathbb{R}^N:a_1\ge a_2\ge\cdots \ge a_N \}.
\end{align}
For $a \in W^{N}$ and $b \in W^{N+1}$ we say that $a$ and $b$ interlace and write $a\prec b$ if:
\begin{align*}
b_1\ge a_1 \ge b_2 \ge a_2 \ge \cdots \ge a_N \ge b_{N+1}.
\end{align*}

We will now consider certain Markov kernels denoted by $\Lambda_{N+1,N}^{\theta}$ from $W^{N+1}$ to $W^N$:

\begin{defn}
Let $\theta \in (0,\infty]$, $N \in \mathbb{N}$.  We define the Dixon-Anderson conditional probability distribution, given by the  Markov kernel $\Lambda_{N+1,N}^{\theta}$ from $W^{N+1}$ to $W^N$, as follows: for fixed $b \in W^{N+1}$, $\Lambda^{\theta}_{N+1,N}(b,\cdot)$ is the distribution of the nonincreasing sequence $a = (a_1, \dots ,a_N)$ of the roots, counted with multiplicity,  of the following 
polynomial of degree $N$: 
$$z \mapsto  \sum_{j=1}^{N+1}  \alpha_j \prod_{1 \leq k \leq N+1, k \neq j}  (z - b_k) $$
where $(\alpha_1, \dots, \alpha_{N+1})$ is Dirichlet distributed, all parameters being equal to $\theta$, if $\theta < \infty$.  For $\theta = \infty$, we take $\alpha_1 = \alpha_2 = \dots = \alpha_{N+1} = 1/(N+1)$, which implies that  $a$ is deterministic and the monic  polynomial with roots $a_1, \dots, a_N$ is $1/(N+1)$ times the derivative of the monic polynomial with roots $b_1, \dots, b_{N+1}$. 
\end{defn}
This definition is not the same as the one used by Assiotis in \cite{Assiotis}: however, it has the advantage to be available even when some of the elements $b_1, \dots, b_{N+1}$ coincide. Moreover, 
it is clear, since the roots of a polynomial are continuous with respect to its coefficients, that the distribution $\Lambda^{\theta}_{N+1,N}(b,\cdot)$ is continuous with respect to $b$. 
If all the $b_j$'s are distinct, we directly recover the definition of  \cite{Assiotis} by applying Proposition 4.2.1 in Forrester \cite{Forrester}: 
\begin{prop}
If $b_1 > \dots > b_{N+1}$ and $\theta < \infty$, the Dixon-Anderson conditional probability distribution defined above is given by: 
\begin{align}\label{DixonAndersonExplicit}
\Lambda^{\theta}_{N+1,N}(b,da)=\frac{\Gamma (\theta (N+1))}{\Gamma(\theta)^{N+1}}\prod_{1\le i <j \le N+1}^{}(b_i-b_j)^{1-2\theta}\prod_{1\le i <j \le N}^{}(a_i-a_j)\prod_{i=1}^{N}\prod_{j=1}^{N+1}|a_i-b_j|^{\theta-1}\mathbf{1}_{\left(a\prec b\right)}\prod_{i=1}^{N}da_i.
\end{align}
\end{prop} 

This distribution had originally been introduced by Dixon at the beginning of the last century in \cite{Dixon} 
and independently rediscovered by Anderson in his study of the Selberg integral in \cite{Anderson}. 

In the sequel we will make essential use of both equivalent forms of the definition of $\Lambda_{N+1,N}^{\theta}$. More precisely, the first form of the definition is used in the analysis performed in Sections 2 and 5 while the second one is mostly used in Section 6. As far as we know, this seems to be a novel aspect of our work as previous papers \cite{GorinMarcus}, \cite{CuencaOrbital}, \cite{Assiotis} appear to use only the second form in display (\ref{DixonAndersonExplicit}) above.

We now define the notion of {\it coherent} or {\it consistent} interlacing arrays, which corresponds to 
random families of arrays following an inhomogeneous Markov chain whose transitions are given by 
 Dixon-Anderson conditional probability distributions: 
\begin{defn}
Let $N\ge 1$ and  $\theta \in (0,\infty]$. 
A coherent, or consistent, random family of interlacing arrays of parameter $\theta$ and length $N$ is a family of random sequences  $\{ a^{(i)}\}_{i=1}^N$, such that $ a^{(i)} \in W^i$, 
\begin{align}
a^{(1)}\prec a^{(2)}\prec \cdots \prec a^{(N-1)} \prec a^{(N)},
\end{align}
and the joint distribution $\mathsf{M}$ of the family satisfies 
\begin{align}
\mathsf{M}\left(da^{(1)},\dots,da^{(N)}\right)= \mu_N (da^{(N)}) \Lambda_{N,N-1}^{\theta}\left(a^{(N)},da^{(N-1)}\right) \cdots \Lambda_{2,1}^{\theta}\left(a^{(2)},da^{(1)}\right),
\end{align}
where $\mu_N$ is the distribution of the top row $a^{(N)}$ of the family. 
\end{defn} 
A useful  particular case of this definition corresponds to the case where the top row is deterministic: 
\begin{defn}
Let $N\ge 1$, $\theta \in (0,\infty]$ and let $\mathsf{a}(N) \in W^N$ be deterministic. An orbital family of interlacing arrays, or orbital beta process, of parameter 
 $\theta$, length $N$ and top row $\mathsf{a}(N)$ is a coherent family of arrays with the same parameter and the same length, such that the top row is almost surely equal to $\mathsf{a}(N)$. 
The law of such a family will be called the orbital distribution of top row $\mathsf{a}(N)$ and parameter $\theta$. 
\end{defn} 
\begin{rmk}
When the coordinates of the top row are all distinct and $\theta$  is finite, 
an  orbital family of interlacing arrays of parameter $\theta$ corresponds to a  $\beta$-corner process as defined in the paper by Gorin and Marcus \cite{GorinMarcus}, for $\beta = 2 \theta$. 
When the parameter $\theta$ is equal to infinity,  an  orbital family of interlacing arrays  of 
length $N$ and top row $\mathsf{a}(N)$ is deterministic (up to an event of probability zero): more precisely, 
 $a^{(i)}$ corresponds to the roots of the $(N-i)$-th derivative of a polynomial whose roots are given by $\mathsf{a}(N)$. 
\end{rmk}

In the present article, we will study the possible distributions of the {\it infinite}  coherent families of interlacing arrays, which are defined as follows: 
\begin{defn}
Let $N\ge 1$ and  $\theta \in (0,\infty]$. A coherent, or consistent, random infinite family of interlacing arrays of parameter $\theta$  is a family of random sequences  $\{ a^{(i)}\}_{i \geq 1}$, such that $ a^{(i)} \in W^i$, 
\begin{align}
a^{(1)}\prec a^{(2)}\prec \cdots \prec a^{(N-1)} \prec a^{(N)} \prec \cdots
\end{align}
and for all $N \geq 1$, $\{ a^{(i)}\}_{i \geq 1}^N$ is a coherent family of interlacing arrays with parameter $\theta$ and length $N$. 
\end{defn}
 
A coherent family of interlacing arrays can be viewed as an inhomogeneous Markov chain, with varying state space,  and moving backwards in discrete time. Its transition probability from level $k+1$ to level $k$ 
is given by $\Lambda_{k+1,k}^\theta$. More precisely, if $E$ is a measurable event with respect to the $\sigma$-algebra generated by the rows of indices larger than or equal to $k+1$, and if $X$ is a Borel subset of $W^k$, we have 
\begin{equation}
\mathbb{P} [ E \cap \{a^{(k)} \in X\}] = \mathbb{E} \left[ \mathbf{1}_{E}  \Lambda_{k+1,k}^\theta (a^{(k+1)},X) \right].  \label{Markovcondition}
\end{equation}
This fact is immediate from the definition in the case of families of finite length. The case of infinite length can then be deduced by using the monotone class theorem. 
This Markov chain point of view will be rather useful in the sequel.

As mentioned earlier, the coherent families of arrays are related to the conjugation-invariant random matrices, because of the  following proposition, which motivates the results of the present paper:   
\begin{prop}
Let $\beta \in \{1,2,4\}$, and let $M$ be a random, finite or infinite, self-adjoint matrix with real entries for $\beta =1$, complex entries for $\beta = 2$, quaternion entries for $\beta = 4$, whose distribution is a central measure, i.e. all the finite square top-left blocks are invariant in law by orthogonal, unitary and symplectic conjugation. Then, the spectra of the successive top-left submatrices of  $M$ form a coherent family of interlacing arrays with parameter $\theta = \beta/2$.  Conversely, 
any coherent family of interlacing arrays with parameter $\theta$ has the same law as the family of spectra of the successive top-left submatrices of a conjugation-invariant matrix whose size is given by the length of the family. 
 \end{prop} 
\begin{proof} 
The first part of the proposition is proven by Neretin in \cite{NeretinTriangles} for finite matrices whose spectrum is deterministic and simple. 
The condition of simple spectrum can be dropped by a continuity arguement. Then, the condition of deterministic spectrum can be removed by conditioning. Finally, the infinite case is immediately deduced from the finite case from the definitions. For finite families of arrays, the converse result is proven by taking a diagonal matrix whose entries are given by the top row, and  by conjugating it with an independent uniform orthogonal, unitary or symplectic matrix.
For infinite families of arrays, we apply Kolmogorov's extension theorem to the distributions of the $N \times N$ matrices obtaind from the  $N$  first rows.
\end{proof} 

The possible probability distributions for infinite  coherent  families of interlacing arrays with a given parameter will be called {\it coherent distributions} or {\it consistent distributions}. They form a convex set, whose extremal points will be called {\it extremal coherent distributions} or  {\it extremal consistant distributions}. Notice that, the analogue of extremal coherent distributions for finite families correspond to the orbital distributions. 

Using Kolmogorov's extension theorem, it is not difficult to check that the coherent distributions of parameter $\theta \in (0,\infty]$ are canonically in bijection with the coherent sequences of probability measures defined as follows:
\begin{defn}
We say that a sequence of probability measures $\{\mu_N\}_{N\ge 1}$ on $\{W^N\}_{N\ge 1}$ is coherent, or consistent for the parameter $\theta \in (0,\infty]$, iff:
\begin{align}
\mu_{N+1}\Lambda_{N+1,N}^{\theta}=\mu_N, \ \forall N \ge 1.
\end{align}
Moreover, we say that such a sequence is extremal iff it cannot be decomposed as a convex combination of two other coherent sequences.
\end{defn}
The canonical  bijection between the coherent distributions on the infinite families of interlacing arrays  and the coherent sequences of probability measures induces a bijection between the extremal points of these respective convex sets.

For $1 \le K \le N$, let us define the  Markov kernel $\Lambda_{N,K}^{\theta}$ from $W^N$ to $W^K$, given by the composition
\begin{align*}
\Lambda^{\theta}_{N,K}=\Lambda^{\theta}_{K+1,K} \circ \cdots \circ \Lambda^{\theta}_{N,N-1}.
\end{align*}

It is clear that a consistent sequence $\{\mu_N\}_{N\ge 1}$ of probability measures satisfies: 
$$ \mu_N \Lambda^{\theta}_{N,K} = \mu_K$$
for $1 \leq K \leq N$. 

As we will see in the main theorem below, the classification of the extremal consistent distributions  is similar for all values of $\theta \in (0,\infty]$, and these measures are indexed by the same set $\Omega$ defined as follows:

\begin{defn} \label{Omegadef} We define the infinite dimensional space $\Omega$ by:
\begin{align}\label{Omegadefinition}
\Omega&=\bigg\{\omega=(\alpha^+,\alpha^-,\gamma_1,\gamma_2)\in \mathbb{R}^\infty \times \mathbb{R}^\infty \times \mathbb{R}  \times \mathbb{R}_+ \big| \nonumber\\
\alpha^+&=(\alpha_1^+\ge \alpha_2^+\ge \cdots \ge 0) \ ; \ \alpha^-=(\alpha_1^-\ge \alpha_2^-\ge \cdots \ge 0);\nonumber\\ & \sum_{i}^{}(\alpha_i^+)^2 + \sum_{i}^{}(\alpha_i^-)^2 < \infty \bigg\}. \nonumber
\end{align}
This space is in bijection with 
\begin{align}
\Omega'&=\bigg\{\omega'=(\alpha^+,\alpha^-,\gamma_1,\delta)\in \mathbb{R}^\infty \times \mathbb{R}^\infty \times \mathbb{R}  \times \mathbb{R}_+ \big| \nonumber\\
\alpha^+&=(\alpha_1^+\ge \alpha_2^+\ge \cdots \ge 0) \ ; \ \alpha^-=(\alpha_1^-\ge \alpha_2^-\ge \cdots \ge 0);\nonumber\\ & \sum_{i}^{}(\alpha_i^+)^2 + \sum_{i}^{}(\alpha_i^-)^2 \leq \delta \bigg\}, \nonumber 
\end{align}
via the correspondence 
$$\delta = \gamma_2 +  \sum_{i}^{}(\alpha_i^+)^2 + \sum_{i}^{}(\alpha_i^-)^2.$$
We endow $\Omega'$ with the topology of point-wise convergence. Then, we endow $\Omega$ with the topology for which the previous bijection is bi-continuous. 
\end{defn}
\begin{rmk}
The topology on $\Omega$ is {\it not} the topology of point-wise convergence: however, it induces the same Borel $\sigma$-algebra. 
\end{rmk}
We also need the following definitions:
\begin{defn}
If $(a^{(i)})_{i \geq 1}$ is a family of interlacing arrays, i.e. $a^{(i)} \in W^{i}$ and 
$a^{(i)}\prec a^{(i+1)}$ for all $i \geq 1$,  the {\it diagonal entries} $(d_i)_{i \geq 1}$ associated to $(a^{(i)})_{i \geq 1}$
are given by $d_1 = a^{(1)}$, and for $i \geq 1$, 
$$d_{i+1} := \sum_{j=1}^{i+1} a^{(i+1)}_j -  \sum_{j=1}^{i} a^{(i)}_j.$$
\end{defn} 
It is easy to check that this definition is consistent with the usual notion of diagonal entries in the case $\beta \in \{1,2,4\}$ where 
we have models of infinite random matrices. 
\begin{defn}
We define the function $\mathfrak{F}_{\omega,\theta}(\cdot)$ on $\mathbb{R}$ for any $\omega \in \Omega$ by:
\begin{align}
\mathfrak{F}_{\omega,\theta}(x)=e^{\i\gamma_1x-\frac{\gamma_2}{2\theta}x^ 2}\prod_{k=1}^{\infty}\frac{e^{-\i\alpha_k^+x}}{\left(1-\i\frac{\alpha_k^+x}{\theta}\right)^{\theta}}\prod_{k=1}^{\infty}\frac{e^{\i\alpha_k^-x}}{\left(1+\i\frac{\alpha_k^-x}{\theta}\right)^{\theta}}
\end{align}
for $\theta < \infty$, and by its limit 
$\mathfrak{F}_{\omega,\infty}(x) = e^{\i\gamma_1}$ for $\theta = \infty$, 
where throughout the paper $\i= \sqrt{-1}$. Note that, this is well-defined since $\sum_{}^{}(\alpha_i^+)^2 + \sum_{}^{}(\alpha_i^-)^2<\infty$. Observe also that $\mathfrak{F}_{\omega,\theta}(\cdot)$ admits a holomorphic extension to the horizontal strip:
\begin{align*}
\bigg\{x \in \mathbb{C}\bigg| |\Im(x)|<\frac{\theta}{\max\{\alpha_1^+,\alpha_1^- \}} \bigg\}.
\end{align*}
\end{defn}

The main result of the present paper gives a full classification of the possible consistent distributions on the infinite families of interlacing arrays, or equivalently, of the possible consistent sequences of probability measures.  
This classification is given by the following two theorems, the first one giving a characterization of  the extremal consistent distributions, and the second one describing the disintegration of the 
general consistent distributions in terms of extremal ones. 
\begin{thm}  \label{main1}
For a given parameter $\theta \in (0, \infty]$, the set of extremal consistent distributions on the infinite families of interlacing arrays  
is in bijection with the set $\Omega$, in such a way 
that the following holds.
If  $(a^{(i)})_{i \geq 1}$ follows the distribution  $\mathsf{M}^{\theta}_{\omega}$  associated with $\omega \in \Omega$,  then the corresponding diagonal entries 
 are i.i.d. with characteristic function $\mathfrak{F}_{\omega,\theta}$. For $\theta < \infty$, 
$$ d_i  = \mathcal{N}_{\gamma_1, \gamma_2/2 \theta} +  \sum_{k=1}^{\infty} \widetilde{\Gamma}_{\theta,\theta/\alpha_k^+}   -   \sum_{k=1}^{\infty} \widetilde{\Gamma}_{\theta,\theta/\alpha_k^-}  $$
where all the variables $\mathcal{N}$ and $\widetilde{\Gamma}$ are independent, 
$\mathcal{N}_{\mu, \sigma^2}$ being Gaussian with mean $\mu$ and variance $\sigma^2$, $\widetilde{\Gamma}_{\theta, \eta}$  being centered gamma distributed  with parameters $\theta$ and $\eta$, i.e. 
it has density
$$t \mapsto \frac{\eta^{\theta}}{ \Gamma (\theta)} \left(t + \frac{\theta}{\eta} \right)^{\theta-1} e^{-\eta 
	 \left(t + \frac{\theta}{\eta} \right) } \mathbf{1}_{t > - \frac{\theta}{\eta}}.$$
	 For $\theta = \infty$, the distribution $\mathsf{M}^{\infty}_{\omega}$ is a Dirac measure. More precisely, if $(a^{(i)})_{i \geq 1}$ follows  $\mathsf{M}^{\infty}_{\omega}$, then 
	almost surely, $d_i = \gamma_1$,  and $(a^{(i)}_j)_{1 \leq j \leq i}$ is given by the roots of the polynomial
	$$z \mapsto z^i + \sum_{j=1}^i c_j \frac{i!}{(i-j)!} z^{i-j}$$
	where the coefficients $c_j$ are determined by the following equality of power series in $z$: 
	$$1 + \sum_{j=1}^{\infty} c_j z^j = e^{-\gamma_1 z - \frac{\gamma_2}{2} z^2} 
\prod_{k \geq 1} e^{z \alpha^+_k} (1 - z \alpha^+_k) \prod_{k \geq 1} e^{-z \alpha^-_k} (1 + z \alpha^-_k),$$
 the limits involved in the infinite products being understood coefficient by coefficient.
\end{thm}

\begin{rmk}
The result above gives the law of $a_1^{(1)}$ but not explicitly the law of  $\left(a_1^{(N)},\dots,a_N^{(N)}\right)$ for $N\ge 2$. However, for any consistent distribution on interlacing arrays of length $N$ and finite parameter $\theta$, its diagonal entries $\left(d_1,\dots,d_N\right)$ uniquely determine it, as we see in Section \ref*{sectionsufficiency}. Moreover, the law of any row $\left(a_1^{(K)},\dots,a_K^{(K)}\right)$ of a random consistent interlacing array under an extremal distribution $\mathsf{M}^{\theta}_{\omega}$ is uniquely determined through its so-called Dunkl transform which is given as a product of functions $\mathfrak{F}_{\omega,\theta}$, see Section \ref{sectionsufficiency} and display (\ref{ExtremalDunklExplicit}) in particular.
\end{rmk}

\begin{rmk}
In Remark 8.3 in \cite{OlshanskiVershik} it is suggested that the method of moments developed in that paper can be used to prove, for general $\theta<\infty$, an equivalent statement (essentially the case $K=1$ of Propositions \ref{necessity} and \ref{sufficiency} below) of Theorem \ref{main1} in the particular case of the bottom entry $a_1^{(1)}$, see Theorem 8.1 in \cite{OlshanskiVershik}. The distribution of the whole consistent interlacing array is not discussed there, however it seems that the approach of \cite{OlshanskiVershik} can be extended to study this as well (we thank the referee for the following interesting observation). A key ingredient in the proof of \cite{OlshanskiVershik}, see Section 5 therein, is an expansion of the multivariate Bessel function for $\theta=1$ (called spherical function in \cite{OlshanskiVershik}) in terms of Schur polynomials. An analogous expansion of Bessel functions for general $\theta<\infty$, with Schur polynomials now replaced by Jack polynomials, was later proven by Okounkov and Olshanski in Section 4 of \cite{OkounkovOlshanskiShifted}. It seems that making use of this formula from \cite{OkounkovOlshanskiShifted} and adapting the work in Section 6 of \cite{OlshanskiVershik} the approach of Olshanski and Vershik can be extended to the case of general $\theta<\infty$. It would still be interesting if these details are worked out somewhere. Finally, as far as we are aware, nothing was known about the $\theta=\infty$ case prior to our work.
\end{rmk}

\begin{thm} \label{main2}
For all $\theta \in (0, \infty]$, and for all probability measures $\nu$ on $\Omega$,  endowed with the topology of point-wise convergence, 
there exists a consistent distribution $\mathsf{M}^{\theta}_{\nu}$ on the infinite families of interlacing arrays, such that 
\begin{equation}
\mathsf{M}^{\theta}_{\nu}[ E]  = \int_{\Omega} \mathsf{M}^{\theta}_{\omega} [E] d \nu (\omega) \label{Mthetanu}
\end{equation}
for all events $E$, the map $\omega \mapsto \mathsf{M}^{\theta}_{\omega} [E] $ being measurable. 
Moreover, the map $\nu \mapsto \mathsf{M}^{\theta}_{\nu}$ is a bijection between the probability measures on $\Omega$ and the consistent distributions
on the infinite families of interlacing arrays, for the parameter $\theta$. 
\end{thm} 

\paragraph{Organisation of the paper} The present article is structured as follows. In Section \ref{sectionthetainfty}, we prove the main theorems in the case where $\theta = \infty$. 
In Section \ref{sectionmainsteps}, we state several propositions which together imply the main theorem for $\theta < \infty$. 
These propositions are proven in Sections \ref{ConvergenceOfOrbital} to \ref{sectiondisintegration}. More precisely, in Section \ref{ConvergenceOfOrbital}, we show that consistent measures on infinite families of interlacing arrays are limits, in a sense which is made precise, of 
orbital distributions. In Section \ref{sectionnecessity}, we show that such convergence of orbital distributions can only occur if the top rows satisfies a particular condition, which has already been stated by Olshanski and Vershik in \cite{OlshanskiVershik} in the case of unitarily invariant ensembles. In Section \ref{sectionsufficiency}, we show that the conditions of Olshanski and Vershik are sufficient to ensure a convergence of the orbital measures 
towards an extremal consistent measure on infinite families of interlacing arrays. In Section \ref{sectiondisintegration}, we show that consistent distribution on infinite families of interlacing arrays can be written as convex combination of 
extremal measures, which corresponds to the statement of Theorem \ref{main2}. In Section  \ref{sectionHuaPickrell}, we discuss the consequences of Theorems \ref{main1} and \ref{main2} for particular $\beta$-ensembles which form 
consistent sequences of probability measures: the $\beta$-Hua-Pickrell and the $\beta$-Bessel point processes. For the $\beta$-Hua-Pickrell process with $s = 0$, we deduce an alternative proof of the result by Killip and Stoiciu \cite{KillipStoiciu} which gives the convergence in law of the point process of the renormalized eigenangles of the Circular beta ensemble when the number of points goes to infinity. 
Moreover, we provide a natural coupling for which a strong convergence occurs, and the result extends to the case of general $s$. Finally, we give an analogous result on the almost sure convergence of the rescaled eigenvalues of the general $\beta$-Laguerre ensemble at the hard edge towards the $\beta$-Bessel point process, that was first proven by Ramirez and Rider in \cite{RiderRamirez}.

\paragraph{Acknowledgements} We thank Vadim Gorin for some useful pointers to the literature. We are grateful to an anonymous referee for a careful reading of the paper and a number of interesting comments and suggestions. The research of T.A. was supported by ERC Advanced Grant 740900 (LogCorRM).

\section{The case $\theta = \infty$} \label{theta=infty} \label{sectionthetainfty}
In order to solve the case $\theta = \infty$ of the main theorem, we will need the following definitions, which will also be useful later. 
\begin{defn}
For $a^{(N)} \in W^N$ we define the quantities 
\begin{align*}
\alpha_{i,N}^+\left(a^{(N)}\right)&=\begin{cases}
\frac{\max\{a^{(N)}_{i},0\}}{N}  \ & i=1,\dots, N\\
0 & i=N+1,N+2,\dots
\end{cases},\\
\alpha_{i,N}^-\left(a^{(N)}\right)&=\begin{cases}
\frac{\max\{-a^{(N)}_{N+1-i},0\}}{N}  \ & i=1,\dots, N\\
0 & i=N+1,N+2,\dots
\end{cases},\\
\gamma_1^{(N)}(a^{(N)})&=\sum_{i}^{}\alpha_{i,N}^+(a^{(N)}) -\sum_{i}^{}\alpha_{i,N}^-(a^{(N)}) = 
\frac{1}{N} \sum_{i=1}^N a^{(N)}_i,\\
\delta^{(N)}(a^{(N)})&=\sum_{i}^{}\left(\alpha_{i,N}^+(a^{(N)})\right)^2+\sum_{i}^{}\left(\alpha_{i,N}^-(a^{(N)})\right)^2 .
\end{align*}
\end{defn}

\begin{defn}
We say that a sequence $\{a^{(N)}\}_{N\ge 1}$ in $\{W^N\}_{N\ge 1}$ satisfies the Olshanski-Vershik (O-V) conditions iff the following limits exist:
\begin{align}
\alpha_i^{\pm}&=\lim_{N \to \infty}\alpha_{i,N}^{\pm}(a^{(N)}), \forall i \ge 1,\\
\gamma_1&=\lim_{N \to \infty}\gamma_1^{(N)}(a^{(N)}),\\
\delta&=\lim_{N \to \infty}\delta^{(N)}(a^{(N)}).
\end{align}
In this case, by Fatou's lemma, 
$$ \gamma_2 := \delta - \left(\sum_{i}^{}(\alpha_i^+)^2 + \sum_{i}^{}(\alpha_i^-)^2 \right)$$
is nonnegative, and 
we will say that $\omega= \left(\alpha^+,\alpha^-,\gamma_1,\gamma_2 \right) \in \Omega$ is the limit point of the sequence $\{a^{(N)}\}_{N\ge 1}$. 
\end{defn}
The O-V conditions might look somewhat artificial at first sight. However, they are more natural when we consider the following result: 
\begin{prop} \label{OVsumsofpowers}
For $a^{(N)} \in W^N$, the O-V conditions are satisfied if and only if for all integers $p \geq 1$, 
the sum of the $p$-th powers of the points of $a^{(N)}/N$ converges to a limit when $N$ goes to infinity. 
\end{prop} 
\begin{proof}
Let us assume that $\{a^{(N)}\}_{N\ge 1}$ satisfies the O-V conditions. The convergence of the sums of the $p$-th powers is satisfied by definition for $p \in \{1,2\}$. For  $p \geq 3$, and $r \geq 1$,  
\begin{align*} \sum_{j=1}^N \left( \frac{ a_j^{(N)}}{N} \right)^p 
& = \sum_{i=1}^{r} \left( \alpha_{i,N}^+\left(a^{(N)}\right)  \right)^p 
+  \sum_{i=1}^{r} \left(-  \alpha_{i,N}^-\left(a^{(N)}\right)  \right)^p  
\\ & + O \left( \sum_{j=1}^N  \left| \frac{ a_j^{(N)}}{N} \right|^p  \mathbf{1}_{ - \alpha_{r+1,N}^-\left(a^{(N)}\right) 
\leq   a_j^{(N)} /N \leq  \alpha_{r+1,N}^+\left(a^{(N)}\right) } \right),
\end{align*} 
which implies 
\begin{align*}
\sum_{j=1}^N \left( \frac{ a_j^{(N)}}{N} \right)^p 
& = \sum_{i=1}^{r} \left( \alpha_{i,N}^+\left(a^{(N)}\right)  \right)^p 
+  \sum_{i=1}^{r} \left(-  \alpha_{i,N}^-\left(a^{(N)}\right)  \right)^p  
\\ & + O \left( \left( \max\left( \alpha_{r+1,N}^+\left(a^{(N)}\right) ,  \alpha_{r+1,N}^-\left(a^{(N)}\right) \right)  \right)^{p-2} 
 \sum_{j=1}^N  \left( \frac{ a_j^{(N)}}{N} \right)^2  \right).
\end{align*} 
The upper and the lower limit of this quantity when $N$ goes to infinity can then both be written as 
$$\sum_{i=1}^{r} \left( \alpha_{i}^+ \right)^p 
+  \sum_{i=1}^{r} \left(-  \alpha_{i}^- \right)^p  
+ O \left(  \delta \left( \max( \alpha_{r+1}^+,  \alpha_{r+1}^-)  \right)^{p-2} \right).$$
Since 
$$\sum_{i=r+1}^{\infty} \left( \alpha_{i}^+ \right)^p 
+  \sum_{i=r+1}^\infty \left(  \alpha_{i}^- \right)^p  
= O \left(  \left( \max( \alpha_{r+1}^+,  \alpha_{r+1}^-)  \right)^{p-2} 
\left(\sum_{i}^{}(\alpha_i^+)^2 + \sum_{i}^{}(\alpha_i^-)^2 \right) \right),$$
we deduce that the upper and the lower limit of the sum of $(a_j^{(N)}/N)^p$ are both equal to 
$$\sum_{i=1}^{\infty} \left( \alpha_{i}^+ \right)^p 
+  \sum_{i=1}^{\infty} \left(-  \alpha_{i}^- \right)^p  
+ O \left(  \delta \left( \max( \alpha_{r+1}^+,  \alpha_{r+1}^-)  \right)^{p-2} \right).$$
Letting $r \rightarrow \infty$, we deduce 
\begin{equation} \sum_{j=1}^N \left( \frac{ a_j^{(N)}}{N} \right)^p  
\underset{N \rightarrow \infty}{\longrightarrow} \sum_{i}^{ } \left( \alpha_{i}^+ \right)^p 
+  \sum_{i}^{} \left(-  \alpha_{i}^- \right)^p \label{sump}
\end{equation} 
for all $p \geq 3$, the last sums being convergent. 
Conversely, if the sum of  $(a_j^{(N)}/N)^p$ converges for all $p \geq 1$, 
the existence of $\gamma$ and $\delta$ is automatically satisfied in the O-V conditions. 
For all polynomials $P$, we have 
$$ \Delta_{\lim} \left(  \sum_{j=1}^N  P \left(  \frac{ a_j^{(N)}}{N}\right)  \right) = 0,$$
where $\Delta_{\lim} $ denotes the difference between the upper and the lower limits when $N$ goes to infinity. 
We deduce that for any function $f$ from $\mathbb{R}$ to $\mathbb{R}$, 
$$ \Delta_{\lim} \left(  \sum_{j=1}^N  f  \left(  \frac{ a_j^{(N)}}{N}\right)  \right)  \leq 
|  \delta | \sup_{|x| \leq S}   \frac{|f(x) - P(x)|}{x^2},$$ 
where 
 $$ S = \sup_{N \geq 1} \max\left(  \alpha_{1,N}^-\left(a^{(N)}\right) ,  \alpha_{1,N}^+\left(a^{(N)}\right)  \right)
\leq  \left( \sup_{N \geq 1} \sum_{j=1}^N  \left( \frac{ a_j^{(N)}}{N} \right)^2  \right)^{1/2}
$$
is finite since the last sum is assumed to converge when $N$ goes to infinity. 
If $f$ is continuous  and equal to zero in a neighborhood of zero, we can uniformly approximate $f(x)/x^2$ by a polynomial $Q$ on the interval $[-S,S]$, and then the supremum of 
$|f(x) - P(x)|/x^2$ on this interval can be made arbitrarily small by taking $P(x) = x^2 Q(x)$. 
Hence, 
\begin{equation}
\sum_{j=1}^N  f  \left(  \frac{ a_j^{(N)}}{N}\right)   \label{sumf}
\end{equation} 
converges when $N$ goes to infinity, for all continuous functions $f$ equal to zero in a neighborhood of zero. 
 Now, let us assume that the upper and the lower limits of $\alpha_{i,N}^+ (a^{(N)})$ do not coincide for some index $i \geq 1$. For $c_1 < c_2$ strictly between the two limits (in particular $c_1$ and $c_2$ are positive), we can consider a continuous, nondecreasing function $f$ equal to zero on $(-\infty, c_1]$ and to $1$ on $[c_2, \infty)$. Then,
the sum   \eqref{sumf} should be larger than or equal to $i$ for infinitely many values of $N$, and smaller than or equal to $i-1$ for infinitely many values of $N$, which contradicts its convergence. 
Hence, $\alpha_{i,N}^+ (a^{(N)})$  is necessarily convergent when $N$ goes to infinity, and the same should occur 
for  $\alpha_{i,N}^- (a^{(N)})$. This proves the O-V conditions.

\end{proof} 

We are now ready to solve the case $\theta = \infty$ of the main theorem. 
In this case, the notion of consistent family of interlacing arrays does not involve any randomness: a family 
  $\{a^{(N)}\}_{N \geq 1}$ is consistent iff  for all $N \geq 1$,  $\{a^{(N-j)}\}_{0 \leq j \leq N-1}$ 
is given by the roots of the successive derivatives of some polynomial of degree $N$. 
The main result of this section is the following:  
\begin{prop} \label{thetainfinitybijection}
For $\theta = \infty$, all  consistent families  of interlacing arrays  satisfy the O-V conditions. Moreover, for all $\omega \in \Omega$, there exists exactly one  consistent family of interlacing arrays whose limit point is $\omega$. 
\end{prop}

\begin{proof}
If $\{a^{(N)}\}_{N \geq 1}$ is consistent, there exists a sequence $(c_j)_{j \geq 1}$ such that for all $N \geq 1$, 
$a^{(N)}$ is given by the roots of the polynomial 
\begin{equation}
z \mapsto \frac{z^N}{N!} + \sum_{j=1}^{N} c_j \, \frac{z^{N-j}}{(N-j)!}. \label{polynomialthetainfty}
\end{equation} 
We deduce that almost surely, each elementary symmetric function of the $N$ points of $a^{(N)}/N$ converges to a limit 
when $N$ goes to infinity. Hence, by Newton's identities, the symmetric functions given by the sums of $p$-th power also converge for all $p \geq 1$, which implies  the O-V conditions by the previous proposition.  
Knowing the limit point of $\{a^{(N)}\}_{N \geq 1}$ implies that we know the limit of the sum of  $(a^{(N)}_j/N)^p$ for all $p \geq 1$: 
\begin{itemize}
\item For $p \in \{1,2\}$, this comes from the definitions. 
\item For $p \geq 3$, this comes from \eqref{sump}.  
\end{itemize}
Hence, we know the limit of each elementary symmetric function of  $a^{(N)}/N$, which determines the coefficients $c_j$, and then the polynomials \eqref{polynomialthetainfty}, and then their roots $\{a^{(N)}\}_{N \geq 1}$.
On the other hand, if $\omega \in \Omega$, one can construct  a sequence $\{\tilde{a}^{(N)}\}_{N \geq 1}$ satisfying the O-V conditions with limit point $\omega$, for example by taking $N \alpha^+_i$ and $N \alpha^-_i$ for $i \leq N^{1/10}$, and each of the two  values $\mu_N - \sqrt{\gamma_2 N}$ and  $\mu_N +  \sqrt{\gamma_2 N}$ a number of times equal to $N/2 + O(N^{1/10})$, for 
$$\mu_N =   \gamma_1 -  \sum_{1 \leq i \leq N^{1/10}}  \alpha^+_i -   \sum_{1 \leq i \leq N^{1/10}}  \alpha^-_i = O(N^{1/10}).$$
We know that  for all $p \geq 1$, the sum of the $p$-th powers of the points in $\tilde{a}^{(N)}/N$ converges to a limit.
Hence, the elementary symmetric functions also converge.
Now, for all $N \geq 1$, the points in $\tilde{a}^{(N)}$ are the roots of some polynomial 
$$ z \mapsto \frac{z^N}{N!} + \sum_{j=1}^{N} \tilde{c}^{(N)}_j \, \frac{z^{N-j}}{(N-j)!},$$
and the convergence of the symmetric functions implies that  $\tilde{c}^{(N)}_j$ tends to a limit $c_j$ when $N \rightarrow \infty$. 
 For $N \geq k \geq 1$, 
the polynomial 
$$ z \mapsto \frac{z^k}{k!} + \sum_{j=1}^{k} \tilde{c}^{(N)}_j \, \frac{z^{k-j}}{(k-j)!}$$
is the $(N-k)$-th derivative of a polynomial whose roots $(\tilde{a}^{(N)}_j)_{1 \leq j \leq N}$ are real: hence, it has real roots. 
Taking the limit of the coefficients when $N$ goes to infinity, we deduce by continuity that 
$$ z \mapsto \frac{z^k}{k!} + \sum_{j=1}^{k} c_j \, \frac{z^{k-j}}{(k-j)!}$$
also has real roots for all $k \geq 1$. Taking the roots of this polynomial for all $k \geq 1$ defines a  new family of interlacing arrays $\{a^{(N)}\}_{N \geq 1}$, for which the renormalized symmetric functions converge to the same limits as for $\{\tilde{a}^{(N)}\}_{N \geq 1}$, since  $\tilde{c}^{(N)}_j$ and $c_j$ have the same limit $c_j$ when $N$ goes to infinity. 
Since  $\{\tilde{a}^{(N)}\}_{N \geq 1}$ has limit point $\omega$, the limits of the sum of the $p$-th powers of the points in $\{a^{(N)}/N\}_{N \geq 1}$  and in  $\{\tilde{a}^{(N)}/N\}_{N \geq 1}$ are both given by $\gamma_1$ for $p =1$, $\delta$ for $p = 2$, and \eqref{sump} for $p \geq 3$. 
The sequence given by  $a^{(N)}/N$ for $N$ odd and  $\tilde{a}^{(N)}/N$ for $N$ even should then satisfy the O-V conditions, which is only possible if $\{a^{(N)}/N\}_{N \geq 1}$  and  $\{\tilde{a}^{(N)}/N\}_{N \geq 1}$ have the same limit point, necessarily equal to $\omega$. 
Hence, $\{a^{(N)}/N\}_{N \geq 1}$ is a consistent family of interlacing arrays with limit point $\omega$. 
\end{proof} 

The proposition we have just proven now implies the main theorem for $\theta = \infty$. Indeed, for this parameter, the deterministic 
consistent families of interlacing arrays are in bijection with $\Omega$, the bijection being given by the limit point. 
Since the notion of consistent families does not involve randomness when $\theta = \infty$, the random consistent families of interlacing arrays 
are in bijection with the random variables with values in $\Omega$. It is easy to deduce the classification of Theorems \ref{main1} and \ref{main2}, where the extremal measure $\mathsf{M}^{\infty}_{\omega}$ is the Dirac measure at the deterministic consistent family whose limit point is $\omega$, and the measure $\mathsf{M}^{\infty}_{\nu}$ is the law of a random consistent family whose limit point follows the 
distribution $\nu$. 
Moreover, if $\{a^{(N)}\}_{N \geq 1}$ is the family corresponding to $\omega$, and if $a^{(N)}$ is given by the roots of 
$$z \mapsto \frac{z^N}{N!} + \sum_{j=1}^N c_j \frac{z^{N-j}}{(N-j)!},$$
then the sum of the points in $a^{(N)}$ is equal to $-Nc_1$, and then $c_1 = - \gamma_1$. We deduce that all the diagonal entries are equal to $-c_1 = \gamma_1$.
It remains to check the formula giving the coefficients $c_j$ in function of $\omega$. 
We have that $(-1)^j c_j$ is the limit of the $j$-th elementary symmetric function of $1/N$ times the points of the $N$-th row, when $N \rightarrow \infty$. 
Now, if $(e_j)_{j \geq 1}$ denotes the elementary symmetric functions of $\lambda_1, \dots, \lambda_r$, we have, by expanding the logarithm, the equality of formal series: 
$$ 1 + \sum_{j=1}^{\infty} (-1)^j e_j z^j = \prod_{q=1}^r (1 - \lambda_q z) 
= \exp \left( - \sum_{k=1}^{\infty} p_k \frac{z^k}{k} \right)$$
where 
$$p_k = \sum_{q=1}^r \lambda_q^k. $$
This identity gives Newton's polynomial relations between the elementary symmetric functions and the sums of successive powers. These relations pass to the limit, so we have the equality of formal series: 
$$1 + \sum_{j=1}^{\infty} c_j z^j  =  \exp \left( - \sum_{k=1}^{\infty} s_k \frac{z^k}{k} \right),$$
where $s_k$ is the limit, when $N$ goes to infinity, of the sum of the $k$-th powers of $1/N$ times the points of the $N$-th row. 
Since the points satisfy the O-V conditions, we have 
$$s_1 = \gamma_1, \, s_2  = \gamma_2 + \sum_{i \geq 1} (\alpha_i^{+})^2 +  \sum_{i \geq 1} (\alpha_i^{-})^2,$$
$$s_k  = \sum_{i \geq 1} (\alpha_i^{+})^k + (-1)^k \sum_{i \geq 1} (\alpha_i^{-})^k$$
for $k \geq 3$. 
We deduce the equalities of formal series: 
$$\log \left( 1 + \sum_{j=1}^{\infty} c_j z^j  \right) 
= - \gamma_1 z - \frac{\gamma_2}{2} z^2 
- \sum_{i \geq 1} \sum_{k=2}^{\infty} \frac{ (z \alpha^+_i)^k + (-z \alpha^-_i)^k}{k},$$
$$\log \left( 1 + \sum_{j=1}^{\infty} c_j z^j  \right)  = - \gamma_1 z - \frac{\gamma_2}{2} z^2 
+ \sum_{i \geq 1}  \left( \log (1 - z \alpha^+_i) + z \alpha^+_i) + \log (1 + z \alpha^-_i) - z \alpha^-_i \right).$$
Taking the exponential completes the proof of Theorem \ref{main1}.

\begin{ex}
If $\gamma_1 = 0$, $\gamma_2 > 0$, and all the $\alpha^+_i$ and $\alpha^-_i$ vanish, then by expanding $e^{-\gamma_2 z^2/2}$, we get $c_{2j} = (-\gamma_2)^j/(2^j j!)$. 
The polynomial corresponding the the $N$-th row is then
$$z \mapsto z^N +  \sum_{1 \leq j \leq N/2} (-\gamma_2)^j \frac{N!}{2^j j! (N-2j)!} z^{N-2j} = \gamma_2^{N/2} H_N(z/\sqrt{\gamma_2}),$$
where $$H_N(z)  = \left( z - \frac{d}{dz} \right)^N (1)$$
is the $N$-th Hermite polynomial. The points of the $N$-th row are then the zeros of the $N$-th Hermite polynomial, multiplied by $\sqrt{\gamma_2}$. From the classification of Theorem \ref{main1}, one can deduce that the shifted and rescaled Hermite polynomials are the only sequences of polynomials which are proportional to successive derivatives of each other, and whose extreme zeros at the $N$-th level are distinct and $o(N)$ when $N$ goes to infinity. 
If  $\gamma_2 = 0$, if a single value of $\alpha^+_i$ or $- \alpha^-_i$ is nonzero, and equal to $\alpha$,  and if $\gamma_1 = \alpha$ then, we have 
to consider $1-\alpha z$, which gives $c_1 = - \alpha$ and $c_j = 0$ for $j \geq 2$. The $N$-th polynomial is then $z^N - N \alpha z^{N-1}$, and at the $N$-th row, we have $N-1$ times zero and one time $N \alpha$. Changing the value of $\gamma_1$ shifts all the points by $\gamma_1 - \alpha$. 
\end{ex}

\section{The main steps of the proof for $\theta < \infty$} \label{sectionmainsteps}

Theorem \ref{main1} for finite $\theta$ will be deduced from the following propositions, proven one by one in the next sections.

\begin{prop}\label{ExtremalityConvergenceOfOrbital}

Let   $\{ a^{(i)}\}_{i \geq 1}$ be a random family of interlacing arrays, which is assumed to follow an extremal consistent distribution with parameter $\theta \in (0, \infty)$. 
For $N \geq K \geq 1$, we consider the restriction to the $K$ first rows of the orbital distribution of top row $a^{(N)}$ and parameter $\theta$. 
Then, this random probability measure on $W^1 \times W^2 \times \dots \times W^K$ almost surely converges in law, as $N$ goes to infinity, to the distribution of 
$\{ a^{(i)}\}_{1 \leq i \leq K}$. 
\end{prop}

\begin{prop} \label{necessity}
Let $\{a^{(i)} \}_{i \ge 1}$ be a sequence such that $a^{(i)} \in W^i$ for all $i \geq 1$, and let $\theta \in (0, \infty)$. 
For $N \geq K \geq 1$, we consider the restriction to the $K$ first rows of the orbital distribution of top row $a^{(N)}$ and parameter $\theta$.
We assume that for all $K \geq 1$,  this probability measure  converges in law to a limiting distribution when $N$ goes to infinity.  Then, 
the sequence $\{a^{(i)} \}_{i \ge 1}$ satisfies the O-V conditions.
\end{prop}

\begin{prop} \label{sufficiency}
Let $\{a^{(i)} \}_{i \ge 1}$ be a sequence such that $a^{(i)} \in W^i$ for all $i \geq 1$. We assume that the O-V conditions are satisfied, for a limit point $\omega \in \Omega$. 
Then, the assumptions of Proposition \ref{necessity} are satisfied, and there exists an extremal consistent distribution $\mathsf{M}^{\theta}_{\omega}$ on the 
infinite families of interlacing arrays, depending only on $\theta$ and $\omega$, such that the limiting distribution involved in the statement of Proposition \ref{necessity} is the restriction of $\mathsf{M}^{\theta}_{\omega}$ to the  $K$ first rows, for all $K \geq 1$. 
Moreover, under $\mathsf{M}^{\theta}_{\omega}$, the diagonal entries 
 are i.i.d. with characteristic function $\mathfrak{F}_{\omega,\theta}$.
\end{prop}
 
 It is easy to prove Theorem \ref{main1} by combining the propositions above. Indeed,  Proposition \ref{ExtremalityConvergenceOfOrbital}  implies in particular that extremal consistent measures are limits of at least one family of orbital measures, in the sense of convergence of finite-dimensional 
marginals. Proposition \ref{necessity} then implies that the top rows of such family of orbital measures should satisfy the O-V condition. Combining this with Proposition \ref{sufficiency} and the fact that distributions on infinite families of interlacing arrays are uniquely determined by their finite-dimensional marginals, we deduce that all extremal consistent measures are of the form $\mathsf{M}^{\theta}_{\omega}$ for some $\omega \in \Omega$. 
Conversely, as we have seen during the proof of Proposition \ref{thetainfinitybijection}, we can construct a sequence satisfying the O-V conditions for any limit point, which shows, from Proposition \ref{sufficiency}, that $\mathsf{M}^{\theta}_{\omega}$ is well-defined and extremal for all $\omega \in \Omega$. The distribution of the diagonal entries is given in  Proposition \ref{sufficiency}, and it
implies that $\omega \mapsto \mathsf{M}^{\theta}_{\omega}$ is injective, and then from the facts just above, it is a bijection between $\Omega$ and  the set of extremal consistent measures 
on random families of interlacing arrays. This completes the proof of Theorem \ref{main1}. 

Theorem \ref{main2} for finite $\theta$ can be deduced from the following propositions: 
\begin{prop} \label{disintegration1}
Fix $\theta \in (0, \infty)$. Then, there exists a measurable space $W^{\infty}$, such that the following holds: 
\begin{itemize}
\item The extremal consistent distribution on the infinite families of interlacing arrays, for the parameter $\theta$,  are in bijection with $W^{\infty}$. 
\item For any consistent distribution $\mathsf{M}$ on the infinite families of interlacing arrays, for the parameter $\theta$, there exists a probability measure $\nu$ on $W^{\infty}$ such that 
for all events $E$, 
$$\mathsf{M} [ E]  = \int_{W^{\infty}} \mathsf{M}_{\mathfrak{w}} [E] d \nu (\mathfrak{w}),$$
where $\mathsf{M}_{\mathfrak{w}}$ is the extremal measure associated with $\mathfrak{w} \in W^{\infty}$ (for the parameter $\theta$), and where the map $\mathfrak{w} \mapsto \mathsf{M}_{\mathfrak{w}} [E] $ is measurable. 
\end{itemize}
\end{prop}
\begin{prop} \label{disintegration2}
In Proposition \ref{disintegration1}, we can take for all $\theta \in (0,\infty)$, $W^{\infty} = \Omega$, $\omega \mapsto \mathsf{M}^{\theta}_{\omega}$ being the bijection involved in the first item. 
\end{prop}
These two propositions immediately imply that all consistent distributions on the infinite families of interlacing arrays are of the form $\mathsf{M}^{\theta}_{\nu}$ for some probability measure $\nu$ on $\Omega$. 
Conversely, if $\nu$ is a probability measure on $\Omega$, it is clear that the formula \eqref{Mthetanu} defines a probability distribution on the infinite families of interlacing arrays, whose consistency is ensured by 
the linearity of \eqref{Markovcondition} with respect to the underlying probability measure. 
Theorem \ref{main2} is then proven, provided that the map $\nu \mapsto \mathsf{M}^{\theta}_{\nu}$ is injective. This is ensured by the following result, which is of interest by itself: 
\begin{thm}\label{ConsistentAlmostSureConvergence}
Let $\theta \in (0, \infty]$, and let $\nu$ be a probability measure on $\Omega$. Then, for an infinite family of interlacing arrays following the distribution $\mathsf{M}^{\theta}_{\nu}$, the successive 
rows a.s. satisfy the O-V conditions, and the corresponding  limit point follows the distribution $\nu$. 
\end{thm} 
\begin{proof}
The case $\theta = \infty$ is a consequence of the discussion at the end of Section \ref{theta=infty}, so we can assume $\theta < \infty$. Moreover, 
by linearity, we can  suppose  that $\nu$ is the Dirac distribution at some $\omega \in \Omega$, which imply that $\mathsf{M}^{\theta}_{\nu} = \mathsf{M}^{\theta}_{\omega}$ is an extremal measure. 
Under this measure, Propositions \ref{ExtremalityConvergenceOfOrbital} and \ref{necessity} together imply that the rows a.s. satisfy the O-V conditions. 
Combining Propositions  \ref{ExtremalityConvergenceOfOrbital} and \ref{sufficiency}, we then deduce that the rows 
follow the distribution  $\mathsf{M}^{\theta}_{\omega'}$, where $\omega'$ is the limit point of the rows. In other words, 
we have $\mathsf{M}^{\theta}_{\omega} = \mathsf{M}^{\theta}_{\omega'}$, and then $\omega = \omega'$, i.e. the limit point of the rows is $\omega$. 
\end{proof}

It is now sufficient to prove the five propositions given in this section. 
Proposition \ref{ExtremalityConvergenceOfOrbital} will  be proven in Section \ref{ConvergenceOfOrbital}, Proposition \ref{necessity} in Section \ref{sectionnecessity}, Proposition \ref{sufficiency} in Section 
\ref{sectionsufficiency}, Propositions \ref{disintegration1} and \ref{disintegration2}  in Section \ref{sectiondisintegration}.

\section{The extremal consistent distributions as limits of orbital distributions} \label{ConvergenceOfOrbital}

In this section, we prove Proposition \ref{ExtremalityConvergenceOfOrbital}, which is an adaptation of Proposition 10.8 in \cite{OlshanskiHarmonic} to our setting, see also \cite{OlshanskiVershik} and the original paper of Vershik \cite{VershikErgodic} where he introduced this so-called 'ergodic method'.

For $N \geq 1$, let us denote by  $\mathcal{F}_{-N}$ the $\sigma$-algebra generated by the random variables $\{ a^{(i)}\}_{i \geq N}$, and let $\mathcal{F}_{-\infty}$ be the intersection of 
$\mathcal{F}_{-N}$ for $N \geq 1$. Let $\mathsf{M}$ be the law of $\{ a^{(i)}\}_{i \geq 1}$, and 
let $A$ be an event in $\mathcal{F}_{-\infty}$. If $k \geq 1$, and if $E$ is an event in $\mathcal{F}_{-k-1}$, we have, from the Markov property satisfied by $\mathsf{M}$
and from the fact that $A \cap E \in \mathcal{F}_{-k-1}$, 
$$\mathsf{M} [ A \cap E \cap \{a^{(k)} \in X\}] = \mathbb{E}^{\mathsf{M}} \left[ \mathbf{1}_{A \cap E}  \Lambda_{k+1,k}^\theta (a^{(k+1)},X) \right]$$
for any Borel set $X \in W^k$. Hence, if $\mathsf{M}(A) > 0$ and if $\mathsf{M}_A$ is the probability measure  given by the restriction of $\mathsf{M}$ to $A$, divided by $\mathsf{M}(A)$, 
we have 
$$\mathsf{M}_A [  E \cap \{a^{(k)} \in X\} ] =  \mathbb{E}^{\mathsf{M}_A} \left[ \mathbf{1}_{E}  \Lambda_{k+1,k}^\theta (a^{(k+1)},X) \right].$$
This shows that $\mathsf{M}_A$ is a consistent distribution on the infinite families of interlacing arrays. 
If $A \in \mathcal{F}_{-\infty}$ is an event with probability strictly between $0$ and $1$, we can apply the result to $A$ and $A^c$, which contradicts the assumption that $\mathsf{M}$ is an extremal measure, 
since 
$$\mathsf{M} = \mathsf{M}(A) \mathsf{M}_A + \mathsf{M}(A^c) \mathsf{M}_{A^c}.$$
Hence, the $\sigma$-algebra $\mathcal{F}_{-\infty}$  is trivial for the probability measure $\mathsf{M}$. 

Let $f$ be a bounded functional which is  measurable with respect to the $\sigma$-algebra generated by $\{ a^{(i)}\}_{1 \leq i \leq K}$, for some finite $K \geq 1$. We have
 that $(\mathbb{E}^{\mathsf{M}} [ f | \mathcal{F}_{-N}])_{N \geq 1}$ is 
a bounded reversed martingale, which then tends a.s. to $\mathbb{E}^{\mathsf{M}} [ f | \mathcal{F}_{-\infty}]$, which is $\mathbb{E}^{\mathsf{M}} [ f ]$ since 
$ \mathcal{F}_{-\infty}$ is trivial. 

For $N \geq K$, let $\mathsf{M}_{a^{(N)}}$ be the random orbital measure with top row $a^{(N)}$ and parameter $\theta$. From the Markov property, one gets a.s.:
$$\mathbb{E}^{\mathsf{M}} [ f | \mathcal{F}_{-N}] = \mathbb{E}^{\mathsf{M}_{a^{(N)}}} [f]$$
 and then 
$$\mathbb{E}^{\mathsf{M}_{a^{(N)}}} [f] \underset{N \rightarrow \infty}{\longrightarrow}  \mathbb{E}^{\mathsf{M}}[f].$$
Combining a suitable countable set of functionals $f$, we deduce that the restriction of $\mathsf{M}_{a^{(N)}}$ to the $\sigma$-algebra generated by $\{ a^{(i)}\}_{1 \leq i \leq K}$ a.s. converges 
to the corresponding restriction of $\mathsf{M}$, which proves Proposition \ref{ExtremalityConvergenceOfOrbital}. 

\section{Necessity of the Olshanski-Vershik conditions for convergence of orbital measures} \label{sectionnecessity}
In this section, we prove Proposition \ref{necessity}. In fact, it will be sufficient to consider the distribution of the bottom entry $a^{(1)}$, given in the following proposition: 
\begin{prop} \label{orbitalDirichlet} 
Under the orbital distribution of top row $(a_1, \dots, a_N)$ and parameter $\theta \in (0,\infty)$, the bottom entry has the distribution of 
$\sum_{j=1}^N \alpha^{(N)}_j a_j$, 
where $(\alpha^{(N)}_1, \dots, \alpha^{(N)}_{N})$ is Dirichlet distributed, all parameters being equal to $\theta$. 
\end{prop} 

\begin{proof}
By Proposition 3.3 of \cite{CuencaOrbital} (see Proposition \ref{OrbitalAndBesselProp} below), the joint Laplace transform of the diagonal entries, under an orbital measure whose top row has distinct elements, is given by a multivariate Bessel function, which 
is symmetric with respect to its arguments (see Theorem \ref{ExistenceBesselTheorem} below). By continuity, the condition that the elements of the top row are distinct can be dropped. 
The diagonal entries are then exchangeable, which in particular implies that the bottom entry has the same law as 
$s_N - s_{N-1}$, where $s_j$ is the sum of the $j$ elements of the $j$-th row. Now, from the definition of the Dixon-Anderson conditional probability distribution, and the 
relation between roots and coefficients of a polynomial, we have the identity in distribution: 
$$s_{N-1} = \left( \sum_{j=1}^N \alpha^{(N)}_j \right)^{-1}  \left( \sum_{j=1}^N \alpha^{(N)}_j \sum_{1 \leq k \leq N, k \neq j} a_k \right)$$
and since $\sum_{j=1}^N \alpha^{(N)}_j = 1$, 
$$s_{N-1} =\sum_{k=1}^N a_k \sum_{1 \leq j \leq N, j \neq k} \alpha^{(N)}_j =  \sum_{k=1}^N a_k ( 1- \alpha^{(N)}_k),$$
which gives 
$$s_N - s_{N-1} = \sum_{k=1}^N a_k \alpha^{(N)}_k.$$
\end{proof}

We will also use the following lemma: 
\begin{lem} \label{modulo} 
If a sequence of real numbers converges modulo any non-zero real number, then it converges. 
\end{lem}
\begin{proof}
If  $(x_n)_{n \geq 1}$ is a sequence converging modulo any non-zero real number, then for all $\lambda \in \mathbb{R}$, $e^{\i \lambda x_n}$ converges to a limit $u_{\lambda}$ of modulus $1$ when $n$ 
goes to infinity. By dominated convergence, for all $\lambda_0 > 0$, and $n$ such that $x_n \neq 0$, 
$$\frac{e^{\i \lambda_0 x_n} - 1 }{ \i x_n} = \int_{0}^{\lambda_0} e^{\i \lambda x_n} d \lambda \underset{n \rightarrow \infty}{\longrightarrow} \int_{0}^{\lambda_0} u_{\lambda} d \lambda.$$
On the other hand, if the sequence $(x_n)_{n \geq 1}$ is unbounded, it is clear that zero is a limit point of the left-hand side, which implies 
that $$\int_{0}^{\lambda_0} u_{\lambda} d \lambda = 0,$$
and then 
$$\int_{0}^{\lambda_0}  (1 + \Re(u_{\lambda}) ) d \lambda = \int_{0}^{\lambda_0}  (1 + \Im(u_{\lambda}) ) d \lambda =  \lambda_0,$$
for all $\lambda_0 > 0$, i.e. the two measures on $\mathbb{R}^*_+$ with densities $1 + \Re(u_{\lambda})$ and $1 + \Im (u_{\lambda})$ are equal to the Lebesgue measure. 
We deduce that for almost every $\lambda > 0$, $\Re(u_{\lambda}) = \Im (u_{\lambda})= 0$, which contradicts the fact that $|u_{\lambda}| = 1$ for all $\lambda \in \mathbb{R}$. Hence, 
$(x_n)_{n \geq 1}$ is in the interval $[-A,A]$ for some $A > 0$. Since  it converges modulo $3A$, it converges. 
\end{proof}
Now, let $\{a^{(i)} \}_{i \ge 1}$ be a sequence satisfying the assumptions of Proposition \ref{necessity}. Taking the particular case $K = 1$, we deduce the convergence in law of 
$\sum_{j=1}^N {a^{(N)}_j} \alpha^N_j$ when $N$ goes to infinity. Equivalently, see for example formula (9.5) in \cite{Kingman}, 
$$\frac{\sum_{j=1}^N {a^{(N)}_j} \Gamma^j_{\theta,1}}{\sum_{j=1}^N \Gamma^j_{\theta,1}}$$
converges in law, where $(\Gamma^j_{\theta,1})_{j \geq 1}$ are i.i.d. Gamma random variables with parameters $\theta$ and $1$, namely with probability density:
\begin{align*}
t\mapsto \frac{1}{\Gamma(\theta)}t^{\theta-1}e^{-t}\mathbf{1}_{t>0}.
\end{align*}
By the law of large numbers and Slutsky's lemma, 
$\sum_{j=1}^N (a^{(N)}_j/N) \Gamma^j_{\theta,1}$ also converges in distribution, which means that 
$$\prod_{j=1}^N  ( 1 - \i \lambda a^{(N)}_j/N )^{-\theta}$$
converges when $N \rightarrow \infty$, for all $\lambda \in \mathbb{R}$, the limit being continuous in $\lambda$.

Taking the squared modulus, we deduce that 
$$ \left(\prod_{j=1}^N  ( 1 +  \lambda^2 (a^{(N)}_j/N)^2) \right)^{-\theta} \leq \left( 1 +   \lambda^2 \sum_{j=1}^N (a^{(N)}_j/N)^2\right)^{-\theta}$$
does not tend to zero for all $\lambda$, which implies that $\sum_{j=1}^N (a^{(N)}_j/N)^2$ is bounded by some constant $K^2 > 0$. Expanding the power $p/2$ for integers $p \geq 1$, 
we deduce 
$$\sum_{j=1}^N |a^{(N)}_j/N|^{p} \leq K^p$$
for $p \geq 2$ even. This  remains true for all reals $p \geq 2$ by using H\"older inequality. 
Taking the logarithm of the convergence of the Fourier transform, we deduce that either 
$$ - \theta \sum_{j=1}^N \log ( 1 - \i \lambda a^{(N)}_j/N) $$
converges in $\mathbb{R} / 2 \i \pi \mathbb{Z}$, or its real part goes to $-\infty$. 
Expanding the logarithm, we get that for $|\lambda| \leq 1/2K$, 
$$ \theta \sum_{p = 1}^{\infty} \frac{(i \lambda)^p}{p} S^{(N)}_p$$
converges in  $\mathbb{R} / 2 \i \pi \mathbb{Z}$, or has  real part going to $-\infty$, for 
$$S^{(N)}_p := \sum_{j=1}^N (a^{(N)}_j/N)^p.$$ 
The real part cannot go to infinity, because $\lambda^p |S_p^{(N)}| \leq (\lambda  K)^{p} \leq 2^{-p} $ for all $p \geq 2$. Hence, we have convergence. 
We also check (again by using that $|\lambda| \leq 1/2K$), that 
$$ \theta \sum_{p = 1}^{\infty} \frac{(\i \lambda)^p}{p} S^{(N)}_p 
= \theta \i \lambda S^{(N)}_1 + O(\lambda^2 K^2),$$
 and then, taking $\lambda = \lambda_0/q$, $q \geq 1$ integer, $|\lambda_0| \leq 1/2K$, and multiplying by $q$,  
$$\theta  \i \lambda_0 S^{(N)}_1 +  O (\lambda_0^2 q^{-1} K^2)$$
converges modulo $2 \i \pi q$, and a fortiori modulo $2 \i \pi$, for all $q \geq 1$ integer. 
The limit points of $ \theta \lambda_0 S^{(N)}_1$ modulo $2 \pi$ are then all in an interval of length $O(\lambda_0^2 q^{-1} K^2)$ for all $q$, and then 
$ \theta \lambda_0 S^{(N)}_1$ converges modulo $2 \pi$. Since $\lambda_0$ is an arbitrary small number, $S^{(N)}_1$ converges modulo any non-zero real number, when $N$ goes to infinity. 
By Lemma \ref{modulo},  $S^{(N)}_1$ converges when $N \rightarrow \infty$. 
We deduce that 
$$ \theta \sum_{p = 2}^{\infty} \frac{(\i \lambda)^p}{p} S^{(N)}_p = \frac{(i \lambda)^2}{2} S^{(N)}_2 +  O (|\lambda|^3 K^3)
$$
converges modulo $2 \i \pi$. Taking the real part, dividing by $\lambda^2$, and letting $\lambda \rightarrow 0$, we deduce that $S_2^{(N)}$ converges when $N$ goes to infinity. 
We now have the convergence modulo $2 \i \pi$ of 
$$ \theta \sum_{p = 3}^{\infty} \frac{(i \lambda)^p}{p} S^{(N)}_p = \frac{(\i \lambda)^3}{3} S^{(N)}_3 + O(\lambda^4 K^4)
$$
which shows that with the notation above, 
$$\frac{(\i \lambda_0)^3}{3} S^{(N)}_3 +  O (\lambda_0^4 q^{-1} K^4)
$$
converges modulo $2 \i \pi q^3$ and then modulo $2 \i \pi$, which implies  $S^{(N)}_3$ converges modulo any non-zero real number, and then converges when $N$ goes to infinity. 
Repeating this procedure, we deduce that  $S^{(N)}_p$ converges  for all integers $p \geq 1$. From 
Proposition \ref{OVsumsofpowers}, the O-V conditions are satisfied. 

\section{Sufficiency of the Olshanski-Vershik conditions for convergence of orbital measures} \label{sectionsufficiency}
In this section, we will prove Proposition \ref{sufficiency}. This will be done in several steps. 
\subsection{Convergence in law of the bottom entry} 
We will first show the convergence in law of the bottom entry.  This is a consequence of the following result: 
\begin{prop} \label{sufficiencyuniformstrip} 
Let $\theta \in (0, \infty)$, and let $\{a^{(i)} \}_{i \ge 1}$ be a sequence such that $a^{(i)} \in W^i$ for all $i \geq 1$, and such that the O-V conditions are satisfied, for a limit point $\omega \in \Omega$. 
Let $d^{(N)}_1$ be the bottom entry (or equivalently, the first diagonal entry) of an orbital family of interlacing arrays with top row $a^{(N)}$ and parameter $\theta$. Then, 
the Laplace transform $y \mapsto \mathbb{E} [ e^{y d^{(N)}_1} ]$  is well-defined everywhere and converges to $y \mapsto \mathfrak{F}_{\omega,\theta}( - \i y)$, uniformly in compact sets of the strip 
\begin{align*}
\mathsf{S}_{\omega,\theta}=\bigg\{y \in \mathbb{C} \bigg | |\Re (y)|<\frac{\theta}{\max \{\alpha_1^+,\alpha_1^-\}}\bigg\}.
\end{align*}
\end{prop} 
\begin{proof}
If we take the notation of the previous section, and if for $p \geq 1$, we denote by $\kappa_p$  the cumulant of order $p$, then we get 
 $$\kappa_p \left( (N \theta)^{-1} \sum_{j=1}^N a^{(N)}_j \Gamma^j_{\theta,1} \right) =   \theta^{-p} S^{(N)}_p \kappa_p ( \Gamma^1_{\theta,1} ).$$
We deduce, from the O-V conditions, 
$$\kappa_1  \left( (N \theta)^{-1} \sum_{j=1}^N a^{(N)}_j \Gamma^j_{\theta,1} \right) \underset{N \rightarrow \infty}{\longrightarrow} \gamma_1,$$
$$\kappa_2  \left( (N \theta)^{-1} \sum_{j=1}^N a^{(N)}_j \Gamma^j_{\theta,1} \right) \underset{N \rightarrow \infty}{\longrightarrow} \theta^{-1} \delta  = \theta^{-1} \left( \gamma_2 + \sum_{i} (\alpha^+_i)^2 +  \sum_{i} (\alpha^-_i)^2 \right)$$
and by using \eqref{sump}, 
$$\kappa_p  \left( (N \theta)^{-1} \sum_{j=1}^N a^{(N)}_j \Gamma^j_{\theta,1} \right) \underset{N \rightarrow \infty}{\longrightarrow}  (p-1)! \theta^{1-p} \left(  \sum_{i} (\alpha^+_i)^p +  \sum_{i} (\alpha^-_i)^p \right)$$
for $p \geq 3$. 
On the other hand, for $x \in \mathbb{C}$ sufficiently close to zero, we get, by expanding the logarithms of the factors of $\mathfrak{F}_{\omega,\theta}$: 
$$\mathfrak{F}_{\omega,\theta}(x) = \exp \left( \i \gamma_1 x - \frac{\gamma_2}{2 \theta} x^2 + \sum_{p \geq 2} \theta^{1-p}  \left(  \sum_{i} (\alpha^+_i)^p +  \sum_{i} (\alpha^-_i)^p \right) \frac{(\i x)^p}{p} \right).$$
This shows that the right-hand side of the convergence statements just above correspond to the cumulants of a random variable $X$ with characteristic function $\mathfrak{F}_{\omega,\theta}$.
Hence, the moments of $(N \theta)^{-1} \sum_{j=1}^N a^{(N)}_j \Gamma^j_{\theta,1} $ converge to the corresponding moments  of $X$. 
Moreover, since the series giving $\mathfrak{F}_{\omega,\theta}$ converges in the neighborhood of zero, $X$ has some exponential moments, and then it is characterized by its moments, which shows that 
$(N \theta)^{-1} \sum_{j=1}^N a^{(N)}_j \Gamma^j_{\theta,1} $ tends to $X$ in distribution. 
Since $(N \theta)^{-1}  \sum_{j=1}^N \Gamma^j_{\theta,1} $ converges a.s. to $1$ by the law of large numbers, Proposition \ref{orbitalDirichlet} and Slutsky's lemma  imply that $d^{(N)}_1$ converges to $X$ in distribution, 
which implies the convergence of $\mathbb{E} [ e^{y d^{(N)}_1} ]$ to $\mathfrak{F}_{\omega,\theta}( - \i y)$ for each $y \in \i \mathbb{R}$. 

Now, $y \mapsto \mathbb{E} [ e^{y d^{(N)}_1} ]$  is a well-defined and entire function for each $N \geq 1$, since $d^{(N)}_1$ is a.s. in the interval between the two extremal entries of $a^{(N)}$. 
Moreover, if we fix $y$ and $y_0$ such that $0 < |y| < y_0 < \frac{\theta}{\max \{\alpha_1^+,\alpha_1^-\}}$, then by the O-V conditions, we have
  $ |y a^{(N)}_j/(N \theta)|  <|y|/y_0 < 1$ as soon as $N$ is large enough, which implies 
$$\mathbb{E} [ e^{y \Gamma^j_{\theta,1} a^{(N)}_j/(N \theta)    } ] = \left( 1 - y a^{(N)}_j/(N \theta) \right)^{-\theta} = \exp \left( y a^{(N)}_j/N+ O_{\theta, y, y_0} \left( \left( a^{(N)}_j/N \right)^2 \right)  \right),$$
and then
 $$\mathbb{E} [ e^{y  \sum_{j=1}^{N} \Gamma^j_{\theta,1}  a^{(N)}_j/(N \theta)   } ]  = \exp \left( y S^{(N)}_1 + O_{\theta, y,y_0} ( S^{(N)}_2 ) \right), $$
 which shows, because of the O-V conditions, that 
 $$\sup_{N \geq 1} \mathbb{E} [ e^{y  \sum_{j=1}^{N} \Gamma^j_{\theta,1}  a^{(N)}_j/(N \theta)   } ]  < \infty.$$
By Proposition \ref{orbitalDirichlet} and the beta-gamma algebra, we deduce that 
 $$\sup_{N \geq 1} \mathbb{E} [ e^{y d^{(N)}_1  (\Gamma^j_{N\theta,1}/ N \theta)} ] < \infty,$$
 and then, applying this result to $y$ and $-y$, 
 $$\sup_{N \geq 1} \mathbb{E} [ \cosh (y d^{(N)}_1  (\Gamma^j_{N\theta,1}/ N \theta)) ] < \infty,$$
 where $\Gamma^j_{N\theta,1}$ is a gamma variable of parameters $N \theta$ and $1$, independent of $d^{(N)}_1$. Since the probability that $\Gamma^j_{N\theta,1}$ is larger that  $N \theta$ is bounded away from zero 
 (it is never zero and it tends to $1/2$ by the central limit theorem), we deduce 
 $$\sup_{N \geq 1} \mathbb{E} [ \cosh (y d^{(N)}_1)   ] < \infty.$$
Since $y$ can be any element strictly smaller than $ \frac{\theta}{\max \{\alpha_1^+,\alpha_1^-\}}$, we deduce that $y \mapsto \mathbb{E} [ e^{y d^{(N)}_1} ]$ is bounded in compact sets of $\mathsf{S}_{\omega,\theta}$, independently 
of $N$. By Montel's theorem, each subsequence of $(y \mapsto \mathbb{E} [ e^{y d^{(N)}_1} ])_{N \geq 1}$ has a sub-subsequence converging uniformly on compact sets of $\mathsf{S}_{\omega,\theta}$
 towards some limiting holomorphic function. From the point-wise convergence on the imaginary axis proven above, the limit is necessarily equal to $y \mapsto \mathfrak{F}_{\omega,\infty}( - \i y)$ on 
 the imaginary axis, and then on all $\mathsf{S}_{\omega,\theta}$ by analytic continuation. This is enough to prove Proposition \ref{sufficiencyuniformstrip}.
 
\end{proof}

\subsection{Multivariate Bessel functions, Dunkl transforms and the orbital beta process}
In order to prove the convergence of the joint distribution of the $K$ first rows for all $K \geq 1$, we will need 
some background and results on the multivariate Bessel functions and Dunkl transforms whose connections with the orbital beta process will be key to our analysis. The references from which we draw the facts stated here are \cite{DeJeu}, \cite{Opdam}, \cite{Dunkl}, \cite{DunklProcesses}, \cite{RoslerRadial} and the exposition in \cite{CuencaOrbital} that we follow at places.
In this section, we fix a parameter $\theta \in (0, \infty)$. 
\subsubsection{Multivariate Bessel functions}
We define the Dunkl operators of type A:
\begin{align*}
T^{\theta}_i=\frac{\partial}{\partial x_i}+\theta \sum_{j\ne i}\frac{1}{x_j-x_i}(1-d_{ij}), \ 1\le i \le N,
\end{align*}
where $d_{ij}$ acts by permuting the variables $x_i$ and $x_j$.
These operators commute for any $i,j$: $T^{\theta}_iT^{\theta}_j=T^{\theta}_jT^{\theta}_i$, see \cite{Dunkl}.

 Now, for any fixed
$a_1,\dots,a_N \in \mathbb{C}$ consider the following system of differential equations (unambiguously defined by commutativity of the $T_i^{\theta}$'s):
\begin{align}\label{DifferentialSystem}
P\left(T^{\theta}_1,\dots, T^{\theta}_N\right)\mathsf{F}(x)=P\left(a_1,\dots, a_N\right)\mathsf{F}(x), \ \textnormal{ for all symmetric polynomials } P.
\end{align}

Then, we have the following theorem, see \cite{Opdam}, \cite{DeJeu}:
\begin{thm}\label{ExistenceBesselTheorem}
For any fixed $a_1,\dots,a_N \in \mathbb{C}$, there exists a unique solution $\mathsf{F}(y_1,\dots, y_N)$ symmetric with respect to the variables $y_1,\dots,y_N$ and normalized by $\mathsf{F}(0,\dots,0)=1$. 
For $a = (a_1,\dots,a_N )$ and  $y = (y_1,\dots,y_N )$, we denote this solution by $\mathfrak{B}_a(y;\theta)$ and call it a multivariate Bessel function. 

Moreover, the map $(a,y)\mapsto \mathfrak{B}_a\left(y;\theta\right)$ admits an extension to a holomorphic function in $2N$ variables and $\mathfrak{B}_{(a_1,\dots, a_N)}(y_1,\dots, y_N)$ is also symmetric with respect to the variables $a_1, \dots, a_N$.
\end{thm} 

We also record the following useful property of $\mathfrak{B}_a(y;\theta)$, see Corollary 3.7 in \cite{CuencaOrbital}. 

\begin{lem}\label{MultiplicativitySymmetryLemma}
For any complex numbers $a_1,\dots,a_N,y_1,\dots,y_N,c \in \mathbb{C}$ we have:
\begin{align*}
\mathfrak{B}_{(ca_1,\dots,ca_N)}(y_1,\dots,y_N;\theta)=\mathfrak{B}_{(a_1,\dots,a_N)}(cy_1,\dots,cy_N;\theta).
\end{align*}
\end{lem}

\subsubsection{Dunkl transforms}
Let $\theta \in (0, \infty)$. Define the weight function $w_N^{\theta}(\cdot)$ on $\mathbb{R}^N$:
\begin{align}
w_N^{\theta}(x_1,\dots,x_N)=\prod_{1\le i <j \le N}^{}|x_i-x_j|^{2\theta}.
\end{align}
Then, we consider a symmetrized version of the Dunkl transform $\mathsf{D}_N^{\theta}$,  defined on functions $f \in L^1\left(\mathbb{R}^N,w_N^{\theta}(x)dx\right)$ as follows (see \cite{RoslerRadial}, \cite{DeJeu}, \cite{DunklProcesses}):
\begin{align}
\mathsf{D}_N^{\theta} \left[f\right](\lambda)=\int_{\mathbb{R}^N}^{}f(x)\mathfrak{B}_{-\i \lambda}\left(x;\theta\right)w_N^{\theta}(x)dx, \ \lambda \in \mathbb{R}^N.
\end{align}
We also define the Dunkl transforms $\mathfrak{D}_N^{\theta}$ and $\mathfrak{E}_N^{\theta}$ for a probability measure $\mu$ on $\mathbb{R}^N$:
\begin{align*}
\mathfrak{D}_N^{\theta}[\mu](y)&=\int_{\mathbb{R}^N}^{}\mathfrak{B}_{\i a}(y;\theta)d\mu(a), \ y \in \mathbb{R}^N,\\
\mathfrak{E}_N^{\theta}[\mu](\lambda)&=\int_{\mathbb{R}^N}^{}\mathfrak{B}_{-\i \lambda}(y;\theta)d\mu(y), \ \lambda \in \mathbb{R}^N.
\end{align*}
The transform $\mathfrak{E}_N^{\theta}$ is a symmetric version of what is usually referred to in the literature as the Dunkl transform for measures. In this work, we will make heavy use of $\mathfrak{D}_N^{\theta}$. It should be thought of as the right analogue of the Fourier transform (equivalently characteristic function) in our setting.

We have the following basic properties of $\mathfrak{D}_N^{\theta}$:

\begin{prop}\label{DunklTransformProperties}Let $\mu$ is a probability measure on $\mathbb{R}^N$, invariant by permutation of the coordinates, and let  $f \in L^1\left(\mathbb{R}^N,w_N^{\theta}(x)dx\right)$.
\begin{itemize} 
\item[(1)]  The function $\mathfrak{D}_N^{\theta}[\mu]$ is continuous and its modulus is bounded by $1$.
\item[(2)] We have: 
\begin{align*}
\int_{\mathbb{R}^N}^{} \mathfrak{D}_N^{\theta}[\mu](y)f(y)w^{\theta}_N(y)dy=\int_{\mathbb{R}^N}^{}\mathsf{D}_N^{\theta}\left[f\right](-\lambda)d\mu(\lambda).
\end{align*}
\item[(3)] The measure $\mu$ is uniquely determined by $\mathfrak{D}_N^{\theta} [\mu]$.
\end{itemize}
\end{prop}

\begin{proof}
Part (1) follows from the estimate in Theorem 2.2 in \cite{DunklProcesses} (which also follows from the representation in Proposition \ref{OrbitalAndBesselProp} below) and the dominated convergence theorem. 

For Part (2), we use Fubini's theorem, which is applicable because of the same estimate: 
\begin{align*}
\int_{\mathbb{R}^N}^{}\mathfrak{D}_N^{\theta}\left[\mu\right](y)f(y)w^{\theta}_N(y)dy&= \int_{\mathbb{R}^N}^{}\int_{\mathbb{R}^N}^{}\mathfrak{B}_{\i \lambda}(y;\theta)d\mu(\lambda)f(y)w^{\theta}_N(y)dy\\
&=\int_{\mathbb{R}^N}^{}\mathsf{D}_N^{\theta} \left[f\right](-\lambda)d\mu(\lambda).
\end{align*}
For Part (3), if $\mathfrak{D}_N^{\theta}[\mu] = \mathfrak{D}_N^{\theta}[\nu] $ for two probability measures $\mu$ and $\nu$, invariant by permutation of the coordinates, then by using Part (2),  we get for any $f \in L^1\left(\mathbb{R}^N,w_N^{\theta}(x)dx\right)$:
\begin{align*}
\int_{\mathbb{R}^N}^{}\mathsf{D}_N^{\theta} \left[f\right](-\lambda)d\mu(\lambda) = \int_{\mathbb{R}^N}^{}\mathsf{D}_N^{\theta} \left[f\right](-\lambda)d\nu(\lambda).
\end{align*}
The  fact that $\mu = \nu$ follows because $\mathsf{D}_N^{\theta}\left[L^1\left(\mathbb{R}^N,w^{\theta}_N(x)dx\right)\right]$ contains 
the symmetric Schwartz space, i.e. the space of functions from $\mathbb{R}^N$ to $\mathbb{C}$ whose derivatives have rapid decay and which are invariant by permutation of the coordinates. 
This last fact is true because the symmetric Schwartz space  is invariant under $\mathsf{D}_N^{\theta}$, due to an inversion formula for this transform: see \cite{CuencaOrbital}, \cite{DeJeu} and \cite{DunklProcesses}. 
\end{proof}

We have the following analogue of Levy's continuity theorem:

\begin{prop}\label{LevyForDunklTransform}
Let $\left(\mu_n\right)_{n \geq 1}$ be a family of probability measures on $\mathbb{R}^N$, invariant by permutation of the coordinates, such that $\mathfrak{D}_N^{\theta}[\mu_n]$ converges pointwise to a function $\phi:\mathbb{R}^N \to \mathbb{C}$ that is continuous at 0 with $\phi \left(0\right)=1 $. Then, there exists a unique probability measure $\mu$ on $\mathbb{R}^N$, invariant by permutation of the coordinates, such that  $\mathfrak{D}_N^{\theta}[\mu]=\phi$, and $\mu_n$ tends weakly to $\mu$.
\end{prop}
\begin{proof}
The same statement where $\mathfrak{D}_N^{\theta}$ is replaced by the usual Dunkl transform and the measures are not supposed to be invariant by permutation corresponds to  
Part (2) of Theorem 2.7 in \cite{DunklProcesses}. If we assume that the measures are invariant by permutation, than we can symmetrize the Dunkl transform without changing the quantities which are involved, and we get our claim 
for  $\mathfrak{E}_N^{\theta}$ instead of  $\mathfrak{D}_N^{\theta}$. The proof for $\mathfrak{D}_N^{\theta}$ is completely analogous. From the properties in Proposition \ref{DunklTransformProperties}, the proof carries over from the classical case: see for example Theorem 23.8 in \cite{BauerBook}.
\end{proof}

\subsubsection{The orbital beta process and Bessel functions}

The connection of multivariate Bessel functions to the orbital beta process is through the following, see \cite{CuencaOrbital}, \cite{GorinMarcus}:

\begin{prop}\label{OrbitalAndBesselProp}
Let $y_1, \dots, y_N \in \mathbb{C}$ and $a \in  W^N$. Then,
\begin{align}
\mathfrak{B}_a\left(y_1,\dots, y_N;\theta\right)=\mathbb{E}\left[\exp \left(\sum_{k=1}^{N} y_k d_k \right)\right]
\end{align}
where  $(d_k)_{1 \leq k \leq N}$ are the diagonal entries of an orbital beta process of parameter $\theta$ and top row $a$. 
\end{prop}
This result has originally  been proven for $a$ with distinct entries, but this condition can be dropped by continuity, since the diagonal entries are uniformly bounded when $a$ is restricted to any compact subset of $W^N$. 
In the next subsection, we will prove a limit theorem for  the function $\mathfrak{B}_{a^{(N)}}$, when  $(a^{(N)})_{N \geq 1}$ is a sequence satisfying the O-V conditions.
For $N \geq m \geq 1$,  $a \in W^N$, $y_1, \dots, y_m \in \mathbb{C}$, we will denote 
$$\mathfrak{B}^N_a\left(y_1,\dots, y_m;\theta\right) := \mathfrak{B}_a\left(y_1,\dots, y_m, 0^{N-m};\theta\right),$$
where $0^{N-m}$ denotes a sequence of $N-m$ coordinates equal to $0$. 

\begin{prop}\label{EquivalentMultidimensionalConvergenceBesselProposition}
Let $K\ge 1$ and suppose $\{a^{(N)}\}_{N\ge 1}$ is a sequence following the O-V conditions,  with limit point $\omega \in \Omega$. Then, 
\begin{align}
\mathfrak{B}_{a^{(N)}}^N\left(y_1,\dots,y_K;\theta\right) \underset{N \to \infty}{\longrightarrow} \prod_{j=1}^{K}\mathfrak{F}_{\omega,\theta}(-\i y_j)
\end{align}
uniformly on compacts in:
\begin{align*}
\mathsf{S}^K_{\omega,\theta}=\bigg\{(y_1,\dots,y_K) \in \mathbb{C}^K \bigg | |\Re (y_i)|<\frac{\theta}{K\max \{\alpha_1^+,\alpha_1^-\}}, \ 1\le i\le K\bigg\}.
\end{align*}
\end{prop} 

Let us now prove that Proposition \ref{sufficiency} can be deduced from Proposition \ref{EquivalentMultidimensionalConvergenceBesselProposition}. Indeed, integrating the result of 
Proposition \ref{OrbitalAndBesselProp} with respect to the distribution of $\tilde{a}^{(m)}$, we get
$$\mathbb{E} [ \mathfrak{B}_{\tilde{a}^{(m)}} \left(y_1,\dots, y_m;\theta\right) ] = \mathbb{E}\left[\exp \left(\sum_{k=1}^{m} y_k d_k \right)\right]$$
where $(d_k)_{1 \leq k \leq m}$ are the $m$ first diagonal entries of any consistent, random family of interlacing arrays of parameter $\theta$ and length at least $m$,  $\tilde{a}^{(m)}$ being its $m^{th}$ row.

For $N \geq K \geq 1$, we apply this to an orbital beta process with top row $a^{(N)}$, successively for $m=N$ and  $m =K$, with $y_k = 0$ when $m = N$ and $k > K$, and we deduce
$$\mathbb{E} [ \mathfrak{B}_{\tilde{a}^{(N,K)}} \left(y_1,\dots, y_K;\theta\right) ] = \mathfrak{B}^N_{a^{(N)}} \left(y_1,\dots, y_K;\theta\right),$$
where $\tilde{a}^{(N,K)}$ is the $K^{th}$ row of this orbital beta process. 
If we denote by $\mu_{N,K}(\cdot)=\Lambda_{N,K}^{\theta}\left(a^{(N)},\cdot\right)$ the distribution of $\tilde{a}^{(N,K)}$ and if  $\mu_{N,K,sym}$ is the law obtained from $\mu_{N,K}$ by 
permuting the coordinates, independently and uniformly at random, we get, using Lemma \ref{MultiplicativitySymmetryLemma}, for all $y \in \mathbb{R}^K$, 
$$\mathfrak{D}_K^{\theta}[\mu_{N,K,sym}] (y)= \int_{\mathbb{R}^K} \mathfrak{B}_{\i a} (y; \theta)  d \mu_{N,K,sym}(a) = \int_{\mathbb{R}^K} \mathfrak{B}_{ a} ( \i y; \theta)  d \mu_{N,K,sym} (a).$$
By symmetry of the multivariate Bessel function with respect to the coordinates of $a$, we can replace $\mu_{N, K,sym}$ by $\mu_{N,K}$  in the last expression, which implies 
\begin{align}\label{ConvergenceOfSymmetrizedDunkl}
\mathfrak{D}_K^{\theta}[\mu_{N,K,sym}] (y) = \mathbb{E} [ \mathfrak{B}_{ \tilde{a}^{(N,K)}} ( \i y; \theta)  ]  =  \mathfrak{B}^N_{a^{(N)}} \left(\i y ;\theta\right) 
\underset{N \to \infty}{\longrightarrow}  \prod_{j=1}^{K}\mathfrak{F}_{\omega,\theta}(y_j).
\end{align}
This convergence implies the convergence of $\mu_{N,K,sym}$ by using 
Proposition \ref{LevyForDunklTransform}. We have then proven the convergence of the distribution of the $K^{th}$ row under the orbital distribution of parameter $\theta$ and top row 
$a^{(N)}$. By the Markov property satisfied by all consistent families of interlacing arrays, we deduce the convergence of the joint distribution of the $K$ first rows to a consistent family of interlacing arrays of length $K$, 
for any finite $K \geq 1$. 
The limiting distributions for different values of $K$ are compatible, which shows, by using Kolmogorov's extension theorem, that they all come from the law of an infinite consistent family of interlacing arrays. 
For this infinite family, the joint law of the $K$ first  diagonal entries is necessarily the limit of the corresponding distribution for the orbital process with top row  $a^{(N)}$. Now, 
by Proposition \ref{OrbitalAndBesselProp}, this distribution has 
a Fourier transform given by $(y_1, \dots, y_K) \mapsto \mathfrak{B}_{a^{(N)}}^N\left(\i y_1,\dots, \i y_K;\theta\right)$, which converges to $\prod_{j=1}^{K}\mathfrak{F}_{\omega,\theta}(y_j)$, and then 
under the limiting infinite family, the diagonal entries are i.i.d. with Fourier transform $\mathfrak{F}_{\omega,\theta}$.
The only result in Proposition \ref{sufficiency}  which remains to be proven is the fact that the law of the limiting infinite family depends only on $\theta$ and $\omega$, and is extremal in the 
set of consistent distributions. 

For the first part of the statement, it is enough to show that the law of an infinite consistent family $(\tilde{a}^{(i)})_{i \geq 1}$ of interlacing arrays is uniquely determined by the joint law of its diagonal entries. 
This is checked as follows: if the law of the diagonal entries is determined, then $\mathbb{E} [ \mathfrak{B}_{\tilde{a}^{(m)}} \left(y_1,\dots, y_m;\theta\right) ]$
is determined for all $m \geq 1$, $y_1,\dots, y_m \in \mathbb{C}$, which fixes the Dunkl transform $\mathfrak{D}_m^{\theta}$ of the symmetrized distribution of $\tilde{a}^{(m)}$, 
and then the law of $\tilde{a}^{(m)}$ by Proposition \ref{DunklTransformProperties}. The joint law of $(\tilde{a}^{(i)})_{1 \leq i \leq m}$ is then fixed by the Markov property, 
which determines the law of  $(\tilde{a}^{(i)})_{i \geq 1}$ since $m$ is arbitrary. 

For the extremality, we proceed as follows: we first observe that taking the distribution of the  diagonal entries induces a map $\rho_1$ from the consistent distributions on the infinite families of interlacing arrays to 
the exchangeable probability measures on the sequences of real numbers, the exchangeability coming from Proposition \ref{OrbitalAndBesselProp} and the fact that the multivariate Bessel functions are symmetric in 
their arguments. It is clear that $\rho_1$ preserves convex combinations of measures, moreover, it is injective since we have seen that the law of an  infinite consistent family of interlacing arrays is uniquely determined by 
the joint law of its diagonal entries. 
Moreover, de Finetti's theorem, see \cite{AldousExchangeability}, provides a natural bijection $\rho_2$ from  the exchangeable probability measures on the sequences of real numbers to the probability measures on the space of probability measures on $\mathbb{R}$, 
this map also preserving convex combinations, and sending the law of i.i.d. random variables with law $\mu$ to the Dirac mass at the probability measure $\mu$. 
Since $\rho_2 \circ \rho_1$ is injective and preserves convex combinations, it sends non-extremal measures to non-extremal measures. In order to prove the  extremality of $\mathsf{M}^{\theta}_{\omega}$,
it is then enough to prove the extremality of $\rho_2 \circ \rho_1(\mathsf{M}^{\theta}_{\omega})$, i.e. the image by $\rho_2$ of the joint law of the diagonal entries under $\mathsf{M}^{\theta}_{\omega}$. 
Now, we have shown that the diagonal  entries are i.i.d., and then $\rho_2 \circ \rho_1(\mathsf{M}^{\theta}_{\omega})$ is the Dirac mass at some probability measure, i.e. it is extremal. 

Finally, for an extremal measure $\mathsf{M}_{\omega}^{\theta}$ we record here a formula for $\mathsf{M}_{\omega,K,sym}^{\theta}$, its symmetrized projection on the $K^{th}$ row of the interlacing array. This is characterized through its Dunkl transform which is given explicitly by, see display (\ref{ConvergenceOfSymmetrizedDunkl}):
\begin{align}\label{ExtremalDunklExplicit}
\mathfrak{D}_K^{\theta}\left[\mathsf{M}_{\omega,K,sym}^{\theta}\right](y_1,\dots,y_K)=\prod_{j=1}^{K}\mathfrak{F}_{\omega,\theta}(y_j).
\end{align}
\subsection{Proof of convergence of the Bessel functions}
In this subsection, we prove Proposition \ref{EquivalentMultidimensionalConvergenceBesselProposition}, which by the discussion just above, completes the proof 
of Proposition \ref{sufficiency}. We start by showing a property of uniform boundedness on compact sets:
\begin{prop} \label{UniformBoundednessMultidimensionalLemma}
Let $K\ge 1$ and suppose that $\{a^{(N)}\}_{N\ge 1}$ is an sequence following the O-V conditions,  with limit point $\omega \in \Omega$. Then, 
$\mathfrak{B}_{a^{(N)}}^N$ is uniformly bounded on any compact set of $\mathsf{S}^K_{\omega,\theta}$, independently of $N \geq K$. 
\end{prop}
\begin{proof} 
By Proposition \ref{OrbitalAndBesselProp}, we have 
$$\mathfrak{B}_{a^{(N)}}^N (y_1, \dots, y_K) = \mathbb{E}\left[\exp \left(\sum_{k=1}^{K} y_k d_k \right)\right]$$
where $d_1, \dots, d_K$ are the $K$ first diagonal entries of an orbital beta process with top row $a^{(N)}$. 
Using H\"older inequality, we get 
$$|\mathfrak{B}_{a^{(N)}}^N (y_1, \dots, y_K)| \leq \prod_{k=1}^K (\mathbb{E} [ e^{K \Re (y_k) d_k} ] )^{1/K}.$$
Since $\mathfrak{B}_{a^{(N)}}^N $ is symmetric with respect to its arguments, $d_1, \dots, d_K$ are exchangeable and a fortiori have the same law, which implies 
$$|\mathfrak{B}_{a^{(N)}}^N (y_1, \dots, y_K)| \leq \prod_{k=1}^K (\mathbb{E} [ e^{K \Re (y_k) d_1} ] )^{1/K}.$$
Now, for $(y_1, \dots, y_K)$ in a given compact set of $\mathsf{S}^K_{\omega,\theta}$, the quantities  $K \Re (y_k) \in \mathsf{S}_{\omega,\theta} $ for $1 \leq k \leq K$ remain in some compact set of 
$\mathsf{S}_{\omega,\theta}$, which gives the  uniform boundedness of $|\mathfrak{B}_{a^{(N)}}^N (y_1, \dots, y_K)|$ by Proposition  \ref{sufficiencyuniformstrip}. 
\end{proof} 
The uniform boundedness of $\mathfrak{B}_{a^{(N)}}^N$ on compact sets of $\mathsf{S}^K_{\omega,\theta}$ shows (by Cauchy's formula) that all the partial derivatives of $\mathfrak{B}_{a^{(N)}}^N$ are also 
uniformly bounded on compact sets, and then the sequence $(\mathfrak{B}_{a^{(N)}}^N)_{N \geq 1}$ is  equicontinuous on compact sets. 
From any subsequence, one can then extract a  sub-subsequence which converges uniformly in compact sets. It is enough to show that the only possible limit of such sub-subsequence 
is $(y_1, \dots, y_K) \mapsto \prod_{j=1}^{K}\mathfrak{F}_{\omega,\theta}(-\i y_j)$.
Now, since the diagonal entries are uniformly bounded for fixed $N$, the functions 
$\mathfrak{B}_{a^{(N)}}^N$ are entire, and then any   limit of a subsequence which is uniform in compact sets of $\mathsf{S}^K_{\omega,\theta}$ should be holomorphic 
in $\mathsf{S}^K_{\omega,\theta}$. By analytic continuation, it is then enough to show that all limits of subsequences should be equal to $\prod_{j=1}^{K}\mathfrak{F}_{\omega,\theta}(-\i y_j)$
for $(y_1, \dots, y_K) $ in some product of segments of the real line which are sufficiently close to zero. 
This is ensured by the following result: 
\begin{prop}\label{PointwiseConvergenceMultidimensionalBesselProp}
Let $K \ge 1$ and suppose $\{a^{(N)}\}_{N\ge 1}$ is an O-V sequence with corresponding point $\omega \in \Omega$. Let $y_1,\dots, y_K$ be real numbers such that:
$$\frac{\theta}{2K \max\{\alpha_1^+, \alpha_1^-\}}>y_1>y_2>\cdots>y_K>0,$$
and for $K \geq 3$, 
$$y_1-y_2>y_3, y_2-y_3>y_4,\dots, y_{K-2}-y_{K-1}>y_K.$$
Then, 
\begin{align}
\lim_{N\to \infty} \mathfrak{B}_{a^{(N)}}^N\left(-y_1,\dots,- y_K\right)\underset{N \to \infty}{\longrightarrow} \prod_{j=1}^{K}\mathfrak{F}_{\omega,\theta}(\i y_j).
\end{align}
\end{prop}
The main tool we use for its proof is the following result, proven in Section 5 of \cite{CuencaOrbital}, that gives an integral expression for a product of two Bessel functions. This is obtained by performing a limit transition of some analogous formulae for the Jack and Macdonald polynomials \cite{CuencaJack}, \cite{CuencaMacdonald}.

\begin{thm}\label{ContourIntegralMultiDimensions}
Let $a \in W^N$ and $m,N \in \mathbb{N}$ such that $1\le m \le N-1$. Consider any real numbers:
\begin{align*}
y_1>y_2>\cdots>y_m>y>0,\\
\min_{i=1,2, \cdots,m-1}(y_i-y_{i+1})>y.
\end{align*}
We write $\mathbf{y}=\left(y_1,\dots,y_m,y\right)$ and $\mathbf{z}=(z_1,\dots,z_m)$. Then, we have:
\begin{align}
& \mathfrak{B}_{a}^N\left(-y_1,\dots,-y_m;\theta\right)\mathfrak{B}^N_{a}\left(-y;\theta\right)=\frac{\Gamma\left(N\theta\right)}{\Gamma\left(\left(N-m\right)\theta\right)\Gamma\left(\theta\right)^{m}}\frac{\prod_{1 \le i<j \le m}^{}\left(y_i-y_j\right)^{1-2\theta}}{y^{m\theta}\left(y_1\cdots y_m\right)^{\theta}} \label{MultidimensionalIntegral}  \\ &
\times \int_{}^{}\cdots \int G^{\theta}_{\mathbf{y}}(\mathbf{z}) F_{\mathbf{y}}(\mathbf{z})^{\theta(N-m)-1} \mathfrak{B}^N_{a}\left(-(y_1+z_1),\dots,-(y_m+z_m),-y+(z_1+\cdots+z_m);\theta\right) \prod_{i=1}^{m}z^{\theta-1}_idz_i 
\nonumber
\end{align}
where the domain of integration $\mathcal{K}_y$ is the compact subset of $\mathbb{R}^m$ defined by the inequalities:
\begin{align*}
z_1,\dots,z_m\ge 0,\\
y\ge z_1+\cdots +z_m,
\end{align*}
and the functions $G^{\theta}_{\mathbf{y}}(\mathbf{z}), F_{\mathbf{y}}(\mathbf{z})$ are given by (with $y_{m+1}=y-\left(z_1+\cdots+z_m\right)$ and $z_{m+1}=0$):
\begin{align*}
G^{\theta}_{\mathbf{y}}(\mathbf{z})&=\prod_{1\le i<j \le m}^{}\left(y_i-y_j+z_i\right)^{\theta-1} \prod_{1 \le i < j \le m+1}^{}\left(y_i-y_j+z_i-z_j\right)\left(y_i-y_j-z_j\right)^{\theta-1},\\
F_{\mathbf{y}}(\mathbf{z})&=\left(1-\left(z_1+\cdots +z_m\right)/y\right) \prod_{i=1}^{m} \left(1+z_i/y_i\right).
\end{align*}
\end{thm}

\begin{proof}[Proof of Proposition \ref{PointwiseConvergenceMultidimensionalBesselProp}]The case $K=1$ is already covered by Proposition  \ref{sufficiencyuniformstrip}, because of Proposition \ref{OrbitalAndBesselProp}. We will deduce the general case by induction on $K$: we assume the result for $K = m \geq 1$ and we deduce it for $K = m+1$. 

We consider real numbers $y_1,\dots,y_m,y$ so that:
$$\frac{\theta}{2(m+1) \max\{\alpha_1^+, \alpha_1^-\}}>y_1>y_2>\cdots>y_m>y>0,$$
and 
$$y_1-y_2>y_3, y_2-y_3>y_4,\dots, y_{m-2}-y_{m-1}>y_m, y_{m-1}-y_m>y$$
if $m \geq 2$. 
We then apply the formula in Theorem \ref{ContourIntegralMultiDimensions}.
From Stirling's formula, we have the following asymptotic for the pre-factor:
\begin{align*}
\frac{\Gamma(\theta N)}{\Gamma((N-m)\theta)}=(N\theta)^{m\theta}(1+O(N^{-1}))
\end{align*}
for fixed $m$ and $\theta$. 
Now, using the induction hypothesis on the left hand side of (\ref{MultidimensionalIntegral}) and Lemma \ref{AuxiliaryLemma1} and Lemma \ref{AuxiliaryLemma2} below on the right hand side, the conclusion of the proposition immediately follows. It thus suffices to prove these two auxiliary results.
\end{proof}

\begin{lem}\label{AuxiliaryLemma1}
Denote by $\mathfrak{S}_{\theta,\mathbf{y},a^{(N)} }(\mathbf{z})$ the integrand in (\ref{MultidimensionalIntegral}) for $a = a^{(N)}$, and let $\frac{1}{2}<u<1$. Then, for fixed 
$u$, $\theta$ and $\mathbf{y}=\left(y_1,\cdots,y_m,y\right)$, 
\begin{align*}
N^{m \theta} \, \bigg| \int \cdots \int_{\mathcal{K}_y \backslash [0,N^{-u}]^m}  \mathfrak{S}_{\theta,\mathbf{y},a^{(N)} } (\mathbf{z})d\mathbf{z}  \bigg|  \underset{N \rightarrow \infty}{\longrightarrow} 0. 
\end{align*}
\end{lem}

\begin{lem}\label{AuxiliaryLemma2}
Denote by $\mathfrak{S}_{\theta,\mathbf{y},a^{(N)}}(\mathbf{z})$ the integrand in (\ref{MultidimensionalIntegral}) for $a = a^{(N)}$, and let $\frac{1}{2}<u<1$. Then, for fixed 
$u$, $\theta$ and $\mathbf{y}$, 

\begin{align*}
 \int_0^{N^{-u}} \cdots \int_0^{N^{-u}} \mathfrak{S}_{\theta,\mathbf{y},a^{(N)}}(\mathbf{z})d\mathbf{z} =\frac{ \left(N\theta\right)^{-m\theta} \Gamma(\theta)^my^{m\theta}\prod_{i=}^{m}y_i^{\theta}}{\prod_{1\le i <j \le m}^{}(y_i-y_j)^{1-2\theta}}
  \left(\mathfrak{B}_{a^{(N)}}^N\left(-\mathbf{y};\theta\right)+O (N^{-u} + N^{1- 2 u})\right).
\end{align*}
\end{lem}

\begin{proof}[Proof of Lemma \ref{AuxiliaryLemma1}]
We clearly have, by continuity of $G$ and the fact that $\mathcal{K}_y$ is compact:
\begin{align*}
\sup_{\mathbf{z} \in \mathcal{K}_y } |G^{\theta}_{\mathbf{y}}(\mathbf{z})|<\infty.
\end{align*}
Moreover, from Proposition \ref{UniformBoundednessMultidimensionalLemma}, we have, since $\left(-(y_1+z_1),\dots,-(y-(z_1+\cdots+z_m))\right)$ remains in some compact subset of $\mathsf{S}^{m+1}_{\omega,\theta}$
when $\mathbf{z} \in \mathcal{K}_y$, 
\begin{align*}
\sup_{N\ge m+1}\sup_{\mathbf{z} \in \mathcal{K}_y} |\mathfrak{B}_{a^{(N)}}^N\left(\left(-(y_1+z_1),\dots,-(y-(z_1+\cdots+z_m))\right);\theta\right)|<\infty.
\end{align*}
Thus, it suffices to show:
\begin{align*}
 \int \cdots \int_{\mathcal{K}_y \backslash [0,N^{-u}]^m}\big|F_{\mathbf{y}}(\mathbf{z}) \big|^{\theta(N-m)-1}\prod_{i=1}^{m}z_i^{\theta-1}dz_i=o(N^{-m\theta}).
\end{align*}
when $N$ goes to infinity. 
We can drop the absolute values since $F_{\mathbf{y}}(\mathbf{z})$ is positive in the range of integration. Moreover, it is easy to show that $F_{\mathbf{y}}(\mathbf{z})$ is monotone decreasing with respect to $z_1,\dots,z_m$. In particular, for any $\mathbf{z}\in\mathcal{K}_y \backslash [0,N^{-u}]^m$, there exists some $1\le i \le m$ such that $z_i>N^{-u}$ and if we let $\mathbf{d}_i$ be the corresponding unit vector in the $i^{th}$ coordinate, then  we have:
\begin{align*}
F_{\mathbf{y}}(\mathbf{z})\le F_{\mathbf{y}}(\mathbf{d}_i)=\left(1-N^{-u}y^{-1}\right)\left(1+N^{-u}y_i^{-1}\right).
\end{align*}
Thus, it suffices to show:
\begin{align*}
 \sum_{i=1}^{m}\left[\left(1-N^{-u}y^{-1}\right) \left(1+N^{-u}y_i^{-1}\right) \right]^{\theta(N-m)-1}\int \cdots \int_{\mathcal{K}_y \cap\{\mathbf{z} \in \mathbb{R}_+^m|z_i\ge N^{-u} \}}\prod_{j=1}^{m}z_j^{\theta-1}dz_j=o(N^{-m\theta}).
\end{align*}
This is clearly true since each integral in the sum is bounded by
\begin{align*}
\int_{0}^{y}\cdots \int_{0}^{y}\prod_{i=1}^{m}z_i^{\theta-1}dz_i=\left(\frac{y^{\theta}}{\theta}\right)^m
\end{align*}
and for $N$ large enough in order to have $\theta(N-m)-1 \geq 0$, 
\begin{align*}
\left[\left(1-N^{-u}y^{-1}\right)\left(1+N^{-u}y_i^{-1}\right) \right]^{\theta(N-m)-1}
\leq \exp \left( N^{-u} (y_i^{-1} - y^{-1}) (\theta(N-m)-1) \right)
\end{align*}
which has rapid decay when $N \rightarrow \infty$ since $y_i^{-1} - y^{-1} < 0$ and $u<1$. 
\end{proof}

\begin{proof}[Proof of Lemma \ref{AuxiliaryLemma2}]
We  have, for $0\le z_i \le N^{-u}$: 
\begin{align*}
\mathfrak{B}_{a^{(N)}}^N\left(-(y_1+z_1),\dots,-(y_m+z_m),-(y-(z_1+\cdots+z_m));\theta\right)=\mathfrak{B}_{a^{(N)}}^N\left(-y_1,\dots,-y_m,-y;\theta\right)\\
+O\left(\sum_{i=1}^{m}|z_i|\sup_{t \in [0,1]}\bigg| \frac{\partial}{\partial z_i}\mathfrak{B}_{a^{(N)}}^N\left(-(y_1+tz_1),\dots,-(y_m+tz_m),-(y-t(z_1+\cdots+z_m));\theta\right) \bigg|\right).
\end{align*}
For fixed $u$, $\theta$ and $\mathbf{y}$, we know that the argument in $\mathfrak{B}_{a^{(N)}}^N$ in the last error term remains in some compact subset of $\mathsf{S}^{m+1}_{\omega,\theta}$ if $N$ is large enough, 
and then the partial derivatives of $\mathfrak{B}_{a^{(N)}}^N$ remain uniformly bounded by Proposition \ref{UniformBoundednessMultidimensionalLemma} and Cauchy integral formula. 
We then get, for $N$ large enough, 
\begin{align*}
& \mathfrak{B}_{a^{(N)}}^N\left(-(y_1+z_1),\dots,-(y_m+z_m),-(y-(z_1+\cdots+z_m));\theta\right) \\ &  =\mathfrak{B}_{a^{(N)}}^N\left(-y_1,\dots,-y_m,-y;\theta\right)+O(N^{-u}).
\end{align*}
Moreover, for $(z_1,\dots,z_m) \in [0,N^{-u}]^m$ we have:
\begin{align*}
G^{\theta}_{\mathbf{y}}(\mathbf{z})=\prod_{i=1}^{m}\left(y_i-y\right)^{\theta}\prod_{1\le i <j \le m}^{}(y_i-y_j)^{2\theta-1}(1+O(N^{-u})), \ F_{\mathbf{y}}(\mathbf{z})^{-\theta m-1}=1+O(N^{-u}).
\end{align*}
Thus, it suffices to prove that:
\begin{align*}
\int_{0}^{N^{-u}}\cdots \int_{0}^{N^{-u}}F_{\mathbf{y}}(\mathbf{z})^{N\theta}\prod_{i=1}^{m}z_i^{\theta-1}dz_i=\frac{\Gamma(\theta)^my^{m\theta}(y_1\dots y_m)^{\theta}}{(N\theta)^{m\theta}\prod_{i=1}^{m}(y_i-y)^{\theta}}\left(1+O(N^{1-2u})\right).
\end{align*}
Towards this end, we use the Taylor expansion around the origin $\mathbf{0}=(0,\dots,0)$ in order to get, for $(z_1,\dots,z_m) \in [0,N^{-u}]^m$, 
$$\log (F_{\mathbf{y}}(\mathbf{z})) = - \sum_{i=1}^m z_i  (y_i^{-1} - y^{-1}) + O(N^{-2u}),$$
and then, since $u > 1/2$, 
\begin{align*}
F_{\mathbf{y}}(\mathbf{z})^{N\theta}=\prod_{i=1}^{m}\exp \left(-\frac{N\theta(y_i-y)}{yy_i}z_i\right) \times \left(1+O\left(N^{1-2u}\right)\right).
\end{align*} 
The claim then follows by observing that:
\begin{align*}
\int_{0}^{N^{-u}}\exp \left(-\frac{N\theta(y_i-y)}{yy_i}z_i\right) & z_i^{\theta-1}dz_i
= \frac{(yy_i)^{\theta}}{(N\theta)^{\theta}(y_i-y)^{\theta}} \int_{0}^{\frac{N\theta(y_i-y)}{yy_i} N^{-u}}\exp \left(- t_i \right) t_i^{\theta-1}dt_i
\\ & = \frac{(yy_i)^{\theta}}{(N\theta)^{\theta}(y_i-y)^{\theta}} \left( \Gamma(\theta) - \int_{\frac{N\theta(y_i-y)}{yy_i} N^{-u}}^{\infty} \exp \left(- t_i \right) t_i^{\theta-1}dt_i \right) \end{align*}
the last integral decaying rapidly when $N \rightarrow \infty$, since $u < 1$. 
\end{proof}

\section{The disintegration of consistent distributions}  \label{sectiondisintegration}
The goal of this section is to prove Propositions \ref{disintegration1} and \ref{disintegration2}. 
For this purpose, we will use abstract and technical results, which are for example developed by Winkler in \cite{Winkler}. 

\begin{proof} [Proof of Proposition \ref{disintegration1}] Let us fix $\theta \in (0, \infty)$. We consider a category $\mathcal{B}$ defined as follows: its objects are given by standard Borel spaces, and  its morphisms are given by the Markov kernels. 
Among the objects, we have  the spaces $(W^N)_{N \geq 1}$, endowed with their Borel $\sigma$-algebra, and for any $N \geq K \geq 1$, we have a morphism $\Lambda_{N,K}^{\theta}$: it is clear that these 
morphisms are compatible, i.e. $\Lambda_{K,L}^{\theta} \circ \Lambda_{N,K}^{\theta} =  \Lambda_{N,L}^{\theta}$ for $N \geq K \geq L\geq 1$. 
We then define the notion of limit object as follows:  
\begin{defn}
 A limit object of $(W^N)_{N \geq 1}$ and $(\Lambda_{N,K}^{\theta})_{N\ge K \geq  1}$ in $\mathcal{B}$  consists of an object $W^{\infty}=\underset{\leftarrow}{\lim} \, W^N$ of $\mathcal{B}$, 
 and Markov kernels (i.e. morphisms) $\Lambda_{\infty,N}^{\theta}: W^{\infty} \to W^N$ such that for $N \geq K \geq 1$, $\Lambda_{N,K}^{\theta} \circ \Lambda_{\infty,N}^{\theta}=\Lambda_{\infty,K}^{\theta}$. Moreover, if an object $\tilde{W}^{\infty}$ of $\mathcal{B}$ and Markov kernels $\tilde{\Lambda}_{\infty,N}^{\theta}:\tilde{W}^{\infty}\to W^N$ satisfy the same condition, then there exists a unique Markov kernel $\Lambda^{\tilde{W}^{\infty}}_{W^{\infty}}:\tilde{W}^{\infty}\to W^{\infty}$ such that  $\tilde{\Lambda}_{\infty,N}^{\theta}=\Lambda^{\theta}_{\infty,N} \circ \Lambda^{\tilde{W}^{\infty}}_{W^{\infty}}$. 
\end{defn}

By a general result of Winkler (see Theorem 4.1.3 in \cite{Winkler}), the limit exists and it is unique up to a Borel isomorphism. 
The space $W^{\infty}$ can be obtained by the following construction. Observe that the Markov kernels $(\Lambda^{\theta}_{N+1,N})_{N \geq 1}$ induce the chain of affine mappings: 
\begin{align*}
\mathcal{M}_p\left(W^{1}\right) \leftarrow\mathcal{M}_p\left(W^{2}\right) \leftarrow \cdots \leftarrow \mathcal{M}_p\left(W^{N}\right)\leftarrow \mathcal{M}_p\left(W^{N+1}\right)\leftarrow \cdots,
\end{align*} 
where $\mathcal{M}_p\left(W^N\right)$ is the simplex of probability measures on $W^N$ equipped with the weak topology. Consider the space $\mathcal{W}=\prod_{N=1}^{\infty}\mathcal{M}_p\left(W^N\right)$ with the product topology and define the inverse system of simplices (not to be confused with the limit in the measurable category $\mathcal{B}$):
\begin{align*}
\underset{\leftarrow}{\lim} \, \mathcal{M}_p\left(W^N\right)=\big\{ \{\mu_N\}_{N\ge 1} \in \mathcal{W}: \mu_{N+1}\Lambda_{N+1,N}^{\theta}=\mu_N \ , \forall N \ge 1\big\},
\end{align*}
consisting of coherent sequences of probability measures. By Theorem 3.2.3  in \cite{Winkler} (see also Corollary 3.2.5 and step 3 in the proof of Theorem 4.1.3 therein), the convex set $\underset{\leftarrow}{\lim} \, \mathcal{M}_p\left(W^N\right)$ is actually a Polish simplex. Moreover, by steps 3 and 4 in the proof of Theorem 4.1.3 on page 103 of \cite{Winkler} (see also second paragraph on page 109 of \cite{Winkler}) its extreme points coincide with $W^{\infty}$; this is exactly how $W^{\infty}$ is constructed. Thus, we have:

\begin{prop} Fix $\theta \in (0, \infty)$. Then, we can take, where $\textnormal{Ex}$ denotes the set of extremal points,
\begin{align}
W^{\infty}=\underset{\leftarrow}{\lim}W^N := \textnormal{Ex}\left(\underset{\leftarrow}{\lim}\mathcal{M}_p\left(W^N\right)\right), \label{Winftycoherent}
\end{align}
and for $\mathfrak{w}\in W^{\infty}$ corresponding to $(\mu_N)_{N\ge 1} \in \textnormal{Ex}\left(\underset{\leftarrow}{\lim}\mathcal{M}_p\left(W^N\right)\right)$, the Markov kernels are given by $\Lambda^{\theta}_{\infty,N}\left(\mathfrak{w},\cdot\right)=\mu_N\left(\cdot\right)$.
\end{prop}

By the fact that the space $\underset{\leftarrow}{\lim}\mathcal{M}_p\left(W^N\right)$ is a Polish simplex, we deduce (see Corollary 3.2.5 in \cite{Winkler}): 
\begin{prop}\label{DisintegrationProp}
Fix $\theta \in (0, \infty)$. For any coherent sequence $\{\mathfrak{M}_N\}_{N\ge 1} \in \underset{\leftarrow}{\lim}\mathcal{M}_p\left(W^N\right)$ there exists a unique probability measure $\nu^{\mathfrak{M}}$ on $W^{\infty}$ such that
for all $N \geq 1$ and all Borel sets $E \subset W^{N}$, 
\begin{align}
\mathfrak{M}_N(E)= \int_{W^{\infty}}^{} d\nu^{\mathfrak{M}}(\mathfrak{w})\Lambda_{\infty, N}^{\theta}\left(\mathfrak{w},E \right), \ \forall N\ge 1, 
\end{align}
the map $\mathfrak{w} \mapsto \Lambda_{\infty, N}^{\theta}\left(\mathfrak{w},E\right)$ being measurable. 
\end{prop}

As we have seen in the introduction, the coherent sequences of probability measures are in canonical bijection with the distributions of consistent infinite families of interlacing arrays. 
This bijection preserves the convex combinations and the extremality. Moreover, if $\mathsf{M}$ is the image of $\{\mathfrak{M}_N\}_{N\ge 1}$, and the extremal measure $\mathsf{M}_{\mathfrak{w}}$  is the  image of the extremal sequence $(\Lambda_{\infty, N}^{\theta}\left(\mathfrak{w},\cdot \right))_{N \geq 1}$, we get 
\begin{equation} \mathsf{M} [ E]  = \int_{W^{\infty}} \mathsf{M}_{\mathfrak{w}} [E] d \nu^{\mathfrak{M}}(\mathfrak{w}), \label{eqME}
\end{equation}
for all events $E$ which are measurable with respect to the $\sigma$-algebra generated by the $N^{th}$ row, $N$ being any positive integer. 
On the other hand, since $ \mathsf{M}_{\mathfrak{w}}$ is the distribution of a consistent infinite family of interlacing arrays for all $\mathfrak{w} \in W^{\infty}$, it is easy to check that 
$$E \mapsto \int_{W^{\infty}} \mathsf{M}_{\mathfrak{w}} [E] d \nu^{\mathfrak{M}}(\mathfrak{w})$$
also defines a consistent distribution. 
Under this distribution and under $\mathsf{M}$, each of the rows has the same law. By the Markov property,  the joint law of finitely many rows is also the same, and by the monotone class theorem,
 \eqref{eqME} holds for all measurable events $E$ on infinite interlacing arrays, 
which proves Proposition \ref{disintegration1}. 
\end{proof}
\begin{proof}[Proof of Proposition \ref{disintegration2}]
Let us assume that $W^{\infty}$ is constructed as in the proof of Proposition \ref{disintegration1}. 
From Theorem \ref{main1} for any $\theta \in (0,\infty)$, we have a bijection $\psi_1 : \omega \mapsto \mathsf{M}^{\theta}_{\omega}$ from $\Omega$ to 
the set $\mathcal{E}$ of extremal consistent distributions on infinite families of interlacing arrays. On the other hand, 
taking the distribution of the successive rows induces a bi-continuous bijection $\psi_2$ from $\mathcal{E}$ to 
$W^{\infty} = \textnormal{Ex}\left(\underset{\leftarrow}{\lim}\mathcal{M}_p\left(W^N\right)\right)$. 
By construction, $\psi_2$ is the bijection involved in Proposition \ref{disintegration1}. 
If we show that $\psi_1$ is a Borel isomorphism, then by bi-continuity of $\psi_2$,  $\psi_2 \circ \psi_1$ is a Borel isomorphism from $\Omega$ to $W^{\infty}$. 
 We can then replace $W^{\infty}$ by $\Omega$ in Proposition  \ref{disintegration1}. The  bijection from $\Omega$ to $\mathcal{E}$
 involved in this proposition is then obtained by composing $\psi_2 \circ \psi_1 : \Omega \rightarrow W^{\infty}$ and 
 $\psi_2^{-1} : W^{\infty} \rightarrow \mathcal{E}$, which gives $\psi_1 : \omega \mapsto \mathsf{M}^{\theta}_{\omega}$. 
 Proposition  \ref{disintegration2} is then proven if we check that $\psi_1$ is a Borel isomorphism from $\Omega$ to $\mathcal{E}$. In fact, it is enough to check that $\psi_1$ is continuous, because by Theorem 3.2 in \cite{Mackey}, a Borel one to one map from a standard Borel space onto
  a subset of a countably generated Borel space is a Borel isomorphism.
  
  In order to get continuity of $\psi_1$, it is enough to show the continuity in $\omega$ of the law of each row under $\mathsf{M}^{\theta}_{\omega}$, 
  because of the continuity properties of the Markov transitions and the monotone class theorem. Using Proposition \ref{LevyForDunklTransform}, 
  we deduce that it is enough to show the continuity in $\omega$ of the joint law of the diagonal entries, and then, since they are i.i.d., the continuity of the law of any diagonal entry. 
  This law has some exponential moments, so it is determined by its cumulants, which have been computed in the proof of Proposition \ref{sufficiencyuniformstrip}: they are given by $\gamma_1$, $\theta^{-1} \delta$ and 
  $$p \mapsto (p-1)!\theta^{1-p} \left(\sum_i (\alpha_i^+)^p + \sum_i (\alpha_i^-)^p \right)$$
  for $p \geq 3$. To prove the continuity of the law with respect to $\omega$, it is enough to check the continuity of the cumulants. For the two first cumulants, continuity is obvious: note that this is due to the fact that 
  we have used the auxiliary space $\Omega'$, for which $\gamma_2$ is replaced by $\delta$,  in the definition of the topology taken on $\Omega$ (see Definition \ref{Omegadef}). 
    For the cumulants of order larger than or equal to $3$, we deduce the continuity from the fact that for all $r \geq 1$, 
 $$\omega \mapsto   \left(\sum_{i \leq r} (\alpha_i^+)^p + \sum_{i \leq r} (\alpha_i^-)^p \right)$$
 is continuous, and  the bound
 \begin{align*}
 \sum_{i > r} (\alpha_i^+)^p + \sum_{i >  r} (\alpha_i^-)^p & \leq ((\alpha_{r+1}^+)^{p-2} + (\alpha_{r+1}^-)^{p-2}) \left(  \sum_{i} (\alpha_i^+)^2 + \sum_{i} (\alpha_i^-)^2 \right)
\\ &  \leq \delta ((\alpha_{r+1}^+)^{p-2} + (\alpha_{r+1}^-)^{p-2}).
\end{align*}

%
 \end{proof}

\section{Consistency and convergence of $\beta$-Hua-Pickrell and $\beta$-Bessel point processes} \label{sectionHuaPickrell}

We recall that we use the parameter $\theta=\beta/2$. For each $N\ge 1$, we define the Hua-Pickrell general $\beta$ ensemble to be the probability measure $\mathfrak{M}^{\theta}_{\mathsf{HP},N,s}$ on $W^N$, depending on a parameter $s \in \mathbb{C}$ such that $\Re s >- \frac{1}{2}$:
\begin{align}
\mathfrak{M}^{\theta}_{\mathsf{HP},N,s}(dx)=\frac{1}{Z^{\theta}_{\mathsf{HP},N,s}}\prod_{j=1}^{N}\left(1+\i x_j\right)^{-s-N\theta}\left(1-\i x_j\right)^{-\bar{s}-N\theta}w^{\theta}_N(x)\mathbf{1}_{(x \in W^N)}dx_1\cdots dx_N.
\end{align}
By using Lemma 2.2 of \cite{NeretinTriangles}, one deduces that the family $( \mathfrak{M}^{\theta}_{\mathsf{HP},N,s})_{N \geq 1}$ is consistent. 
This family induces a consistent distribution  $\mathsf{M}^{\theta}_{\mathsf{HP},s}$ on the infinite families of interlacing arrays. 
By Theorem \ref{main2}, we deduce that there exists a probability measure $\nu^{\mathsf{HP},\theta,s}$ on $\Omega$ such that 
$\mathsf{M}^{\theta}_{\mathsf{HP},s} = \mathsf{M}^{\theta}_{\nu^{\mathsf{HP},\theta, s}} $.
Hence, from Theorem \ref{ConsistentAlmostSureConvergence} under $\mathsf{M}^{\theta}_{\mathsf{HP},s}$, the successive rows a.s. satisfy the O-V conditions.
In particular, we get almost sure convergence of the extremal points, divided by $N$, towards some limiting point process. 
For $\theta = 1$ (i.e. $\beta = 2$) and $s=0$, this result gives the almost sure convergence of the renormalized eigenangles of a virtual isometry following the Haar distribution: see \cite{NeretinUnitary} and \cite{BourgadeNajnudelNikeghbali}. 
The almost sure convergence of the renormalized extremal points implies the  convergence in distribution of the corresponding point processes. For general $\theta$ and $s = 0$, after applying the Cayley transform, we deduce the 
convergence of the point process of the renormalized eigenangles of the Circular beta ensemble towards a limiting point process. This result has already been proven by Killip and Stoiciu in \cite{KillipStoiciu}: our method thus gives an alternative proof, which is less explicit in the description of the limiting point process, but which has the advantage of giving a natural coupling such that strong convergence occurs. 
The limiting process has been interpreted as the spectrum of a random operator in a paper by Valk\'o and Vir\'ag \cite{ValkoVirag}. 
For general $\theta$ and $s$, we also deduce  a convergence in law of point processes, for which a proof had already been announced in \cite{ValkoVirag}, the limiting process being  again obtained as the spectrum of an operator. 

Another question concerns the distribution of the parameters $\gamma_1$ and $\gamma_2$  when $\omega$ follows the distribution $\nu^{\mathsf{HP},\theta,s}$.
We conjecture that almost surely, $\gamma_2 = 0$ and 
$$\gamma_1 = \underset{m \rightarrow \infty}{\lim} \left( \sum_{i} \alpha_i^+ \mathbf{1}_{\alpha_i^+ \geq m^{-2}}  -  \sum_{i} \alpha_i^- \mathbf{1}_{\alpha_i^- \geq m^{-2}} 
\right).$$
This result has been proven by Qiu in \cite{Qiu} when $\theta = 1$ (i.e. $\beta =2$), for all $s$ with $\Re s >- 1/2$  in the case of $\gamma_2$, and for  $s \in (-1/2, \infty)$ in the case of $\gamma_1$. Note that in \cite{Qiu}, similar results have been proven for the ergodic decomposition of infinite Hua-Pickrell measures, which can be defined when $\Re s \leq -1/2$: such decomposition of infinite measures has been proven to make sense by Bufetov in \cite{Bufetov}. We do not know if some of the results of the present paper can be extended to the case of infinite $\sigma$-finite consistent measures on the space of  infinite families of interlacing arrays.

Moreover, for each $N\ge 1$, we can also define the inverse Wishart/Laguerre general $\beta$ ensemble as the probability measure $\mathfrak{M}^{\theta}_{\mathsf{IW},N,\tau}$ on $W_+^N = W^N \cap \mathbb{R}_+^N$,  depending on a parameter $\tau > -1$:
\begin{align}
\mathfrak{M}^{\theta}_{\mathsf{IW},N,\tau}(dx)=\frac{1}{Z^{\theta}_{\mathsf{IW},N,\tau}}\prod_{j=1}^{N}x_j^{-\tau-2\theta N}e^{-\frac{2}{x_j}}w^{\theta}_N(x)\mathbf{1}_{(x \in W_+^N)}dx_1\cdots dx_N.
\end{align}
The computation of the integral in 'Variant A' in Section 2.2 of \cite{NeretinTriangles}, after  a change of variables $x \mapsto \frac{1}{x}$, proves that 
the family $(\mathfrak{M}^{\theta}_{\mathsf{IW},N,\tau})_{N \geq 1}$ is consistent. We deduce, from Theorem \ref{ConsistentAlmostSureConvergence}, similar convergence results as in the case of Hua-Pickrell measures, for the renormalized largest eigenvalues. Since with the change of variables $x \mapsto \frac{1}{x}$ the extremal points under $\mathfrak{M}^{\theta}_{\mathsf{IW},N,\tau}$ become the smallest eigenvalues of the general $\beta$-Laguerre ensemble (and the rescaling $\frac{1}{N}\mapsto N$ is exactly the hard edge scaling) we obtain the almost sure convergence towards the $\beta$-Bessel point process, which is described through the generator of a random diffusion, see \cite{RiderRamirez}.

The distribution of the parameters $\gamma_1$ and $\gamma_2$ in the case of the inverse Wishart measures for $\theta=1$ (i.e. $\beta=2$) has recently been determined in \cite{AssiotisInverseWishart}: almost surely $\gamma_2=0$ and (recall that $\alpha^-_j\equiv 0$) $\gamma_1=\sum_{i}^{}\alpha_i^+$. It is natural to expect that this result holds for any $\theta>0$.

\bigskip
\noindent
{\sc School of Mathematics, University of Edinburgh, James Clerk Maxwell Building, Peter Guthrie Tait Rd, Edinburgh EH9 3FD, U.K.}\newline
\href{mailto:theo.assiotis@ed.ac.uk}{\small theo.assiotis@ed.ac.uk}

\bigskip
\noindent
{\sc School of Mathematics, University of Bristol, U.K.}\newline
\href{mailto:T.Assiotis@bristol.ac.uk}{\small Joseph.Najnudel@bristol.ac.uk}

\end{document}